\documentclass[11pt,reqno]{amsart}

\usepackage{amssymb,amsmath,amsthm,amsfonts}
\usepackage{setspace}
\usepackage{enumerate}
\usepackage{amsaddr}
\usepackage{latexsym}
\usepackage{color}
\usepackage{hyperref}
\hypersetup{
  colorlinks   = true,    
  urlcolor     = blue,    
  linkcolor    = blue,    
  citecolor    = red      
}
\usepackage{graphicx}

\newtheorem{thm}{Theorem}[section]
\newtheorem{prop}[thm]{Proposition}
\newtheorem{lem}[thm]{Lemma}
\newtheorem{cor}[thm]{Corollary}
\newtheorem{definition}[thm]{Definition}
\newtheorem{remark}[thm]{Remark}

\newcommand{\R}{\mathbb{R}}
\newcommand{\Q}{\mathbb{Q}}
\newcommand{\T}{\mathbb{T}}
\newcommand{\p}{\partial}
\newcommand{\lb}{\langle}
\newcommand{\rb}{\rangle}
\renewcommand{\d}{\operatorname{d}\!}

\allowdisplaybreaks

\begin{document}

\title[Dynamics of the Majda-Biello System]{Smoothing and Global Attractors for the Majda-Biello System on the Torus}

\author{E.~Compaan}
\address{Department of Mathematics, University of Illinois, Urbana, IL}
\thanks{The author was supported by a National Physical Science Consortium fellowship. \\ \emph{Email address:} compaan2@illinois.edu}

\begin{abstract}
In this paper, we consider the Majda-Biello system, a coupled KdV-type system, on the torus. In the first part of the paper, it is shown that, given initial data in a Sobolev space, the difference between the linear and the nonlinear evolution almost always resides in a smoother space. The smoothing index depends on number-theoretic properties of the coupling parameter in the system which control the behavior of the resonant sets. In the second part of the paper, we consider the forced and damped version of the system and obtain similar smoothing estimates. These estimates are used to show the existence of a global attractor in the energy space. We also show that when the damping is large in relation to the forcing terms, the attractor is trivial. 
\end{abstract}

\maketitle


\section{Introduction}

This paper studies the following system of coupled KdV-type equations on the torus
\begin{equation} \label{eq:MB}
\begin{cases}
u_t + u_{xxx} + \frac12 (v^2)_x = 0, \quad x \in \mathbb{T}\\
v_t + \alpha v_{xxx} + (uv)_x = 0, \\
u_0,v_0 \in H^s(\T).
\end{cases}
\end{equation}

This system was introduced  by Majda and Biello, \cite{MB}, \cite{Biello}, as a simplified asymptotic model for the behavior of certain atmospheric Rossby waves. Rossby waves are long atmospheric or oceanic waves which have important effects on weather patterns and ocean currents. The system \eqref{eq:MB} models such waves in the upper atmosphere.  In the model, $u$ corresponds to a Rossby wave with significant energy in the midlatitudes and $v$ corresponds to a Rossby wave confined to the equatorial region. The system is designed to capture the nonlinear interactions between the waves under specific physical conditions -- such interaction is relevant in both theoretical atmospheric science and weather prediction. Majda and Biello obtained numerical estimates of $0.899$, $0.960$, and $0.980$ for the coupling parameter $\alpha$ in the physical cases they considered. We note that in the case of atmospheric waves, the periodic problem is physically relevant. 

Solutions of the Majda-Biello system have momentum conservation. They also satisfy conservation laws at the $L^2$ and $H^1$ levels. Specifically, the following quantities are constant: 
\begin{equation} \label{eq:conservationlaws} \textstyle
E_1 =\int u \; \d x \qquad E_2 = \int v \; \d x \qquad E_3 = \int u^2 + v^2 \; \d x \qquad E_4 = \int u_x^2 + \alpha v_x^2 - uv^2 \; \d x. 
\end{equation}

The last integral above is the Hamiltonian conservation law. However, unlike the KdV, the system is not completely integrable, even in the relatively simple case $\alpha = 1$. It was recently shown by Vodov\'{a}-Jahnov\'{a} that there are no higher conservation laws \cite{Vod}. The system scales like the KdV, leading to a critical Sobolev index of $-\frac32$.

Coupled KdV-type systems have been extensively studied, see e.g. \cite{GST}, \cite{ACW}, \cite{LP}, \cite{ST}, \cite{AK}, but little of the work addresses periodic problems with coupling parameter $\alpha \neq 1$ such as appears in \eqref{eq:MB}.  For the Majda-Biello system on $\R$, and systems with similar coupling, more is known. For the related Gear-Grimshaw system \cite{GG}, a model of gravity waves in stratified fluids, Bona, Ponce, Saut, and Tom proved local well-posedness results in $H^s(\R) \times H^s(\R)$ for $s \geq \frac34$ \cite{BPST}. In \cite{Feng}, the same result for the Hirota-Satsuma system, another similar coupled KdV system, is proven. In \cite{Oh}, Oh proved global well-posedness for the Majda-Biello system on $\R$ with $s \geq 0$.  

The well-posedness of \eqref{eq:MB} on $\T$ was also studied in \cite{Oh}, and local well-posedness in $H^s$ for $s$ above a threshold $s^*$ established. The value of $s^*$ is dependent on the arithmetic properties of $\alpha$, leading to well-posedness results of markedly different types depending on the nature of $\alpha$. When $\alpha =  1$, the resonant interactions in the system simplify significantly. In this case, the methods used by Kenig, Ponce, and Vega in \cite{KPV} to prove the local well-posedness of the KdV equation can be applied; see \cite{Oh}. This gives local well-posedness in $H^{-\frac12} \times H^{-\frac12}$ for mean zero initial data. A further argument gives the result for general initial data \cite{Oh}. Oh also shows that for $\alpha < 0$ and $\alpha > 4$, the resonant interactions are easier to control and the KdV theory can be applied.

For $\alpha \in (0,1) \cup (1, 4] $, the behavior is more complex. Oh used the restricted norm method of Bourgain \cite{Bourg} to prove local well-posedness in $H^s \times H^s$ for $s \geq \min\big\{1, \frac12 + \frac 12 \max\{\nu_c, \nu_d\}+\big\}$ with the assumption that the initial data $u_0$ is mean zero. The values $\nu_c$ and $\nu_d$ are number-theoretic parameters which depend on the properties of $\alpha$; generically $\nu_c = \nu_d = 0$ for almost every $\alpha$. Introducing these parameters gives control over the resonant sets which arise in Bourgain space estimates. For any $\alpha$, local well-posedness extends to global for $s \geq 1$ due to conservation of the Hamiltonian $E_4$. This implies that the system is globally well-posed in $H^s$ for $s \geq 1$ regardless of the value of $\alpha$. In \cite{Oh2},  global well-posedness for $s > s^*(\alpha) \geq \frac57$ was established using the I-method. Here again, the threshold value depends on properties of $\alpha$. In the special case $\alpha = 1$, global well-posedness holds for $s> - \frac12$. 

This paper is concerned with the dynamics of solutions to the Majda-Biello system. In the first part, we demonstrate that the difference between the linear evolution and the nonlinear evolution resides in a higher-regularity space. The result follows from a combination of the method of normal forms of Babin, Ilyin, and Titi \cite{BIT} and the restricted norm method. This approach was first used by Erdo\u{g}an and Tzirakis in \cite{ET1} and \cite{ET2} on the KdV and the Zakharov system. The difficulty in applying their methods to this particular system comes from the complexity of the resonance relations. The coupling of the equations through $\alpha$ makes the resonances significantly more complex than those of the KdV and the Zakharov system. Unlike the KdV case, the resonance equations do not factor neatly, and the coupling interactions are considerably more difficult to control than those of the Zakharov system. 

The normal form transformation eliminates the derivative nonlinearity and replaces it with a third-order power nonlinearity. Controlling this requires trilinear $X^{s,b}$ estimates, in contrast to the bilinear estimates necessary for well-posedness. The local theory used multipliers of the form
\[\frac{k \lb k\rb ^s \lb k_1 \rb^{-s} \lb k_2\rb^{-s}}{\lb \tau - k^3 \rb^{1-b}\lb \tau_1 - \alpha k_1^3 \rb^{1/2} \lb \tau_2- \alpha k_2^3 \rb^{1/2}}, \]
whereas the smoothing results require control over multipliers such as
\begin{equation*}
\frac{k (k_1+ k_2)\lb k \rb^{s_1}\lb k_1 \rb^{-s} \lb k_2 \rb ^{-s} \lb k_3\ \rb^{-s}}{(k^3 - \alpha(k_1+ k_2)^3 - \alpha k_3^3 )\lb \tau - k^3 \rb^{1-b}\lb \tau_1 - k_1^3 \rb^{1/2} \lb \tau_2 - \alpha k_2^3 \rb^{1/2} \lb \tau_3 -  \alpha k_3^3 \rb ^{1/2}}.
\end{equation*}
For the latter, we want $s_1 > s$ to obtain smoothing. This means we have no \emph{a priori} bound on $\lb k \rb^{s_1}\lb k_1 \rb^{-s} \lb k_2 \rb ^{-s}\lb k-k_1-k_2 \rb^{-s}$. Furthermore, the differentiation by parts introduces the term $k^3 - \alpha(k_1+ k_2)^3 - \alpha(k-k_1-k_2)^3$ in the denominator. Unlike the bracketed terms which appear in the local theory multiplier, this can be arbitrarily small. The estimates require precise control of multiple terms to ensure that the multiplier remains bounded. Depending on the characteristics of $\alpha$, we obtain different levels of smoothing, with a gain of up to $\frac12$ for $\alpha \neq 1$. Again, the results improve if $\alpha = 1$; the KdV results in \cite{ET1} can be applied to get a gain of up to $1$ derivative.

In the second part of the paper, we consider the behavior of the system when forcing and weak damping terms are included:
\begin{equation} \label{eq:MBforced}
\begin{cases}
u_t + u_{xxx} + \gamma u + \frac12 (v^2)_x = f \\
v_t + \alpha v_{xxx} + \delta v + (uv)_x = g  
\end{cases}
\end{equation}
We take initial data $u_0,\; v_0 \in H^1$; the functions $f$ and $g$ are in $H^1$ with mean zero and the coefficients $\gamma$ and $\delta$ are positive. We investigate the long-time dynamics of this equation, and show that for almost every $\alpha$, all solutions will eventually enter a compact set, the global attractor, which is an invariant set of the evolution. The existence of such sets has been studied extensively, particularly for dissipative systems. For the KdV, global attractors were first studied by Ghidaglia in $H^2$ \cite{Ghid}. Further work by other authors has established the existence below the $L^2$ level; see the discussion and references in \cite{ET3}. To obtain an attractor for the Madja-Biello system, we use the method of \cite{ET3} and \cite{ET2} along with our smoothing estimate to decompose the solution into two parts: the linear part which decays over time thanks to the damping terms, and the nonlinear part. We then apply smoothing estimates to the nonlinear part to show that it resides in a smoother space. This gives a global attractor for almost every $\alpha \in (0,1)$. For $\alpha = 1$, the estimates in \cite{ET3} can be applied directly and one can obtain an $L^2$ attractor. 

One reason for the interest in global attractors is that they can be finite-dimensional even when the phase space of the equation is not, making them useful tools in understanding the dynamics of a system. In the last part of the paper, we show that the attractor for the Majda-Biello system is trivial, consisting of a single pair of functions $(p,q) \in H^5 \times H^5$, if the damping coefficients $\delta$ and $\gamma$ are sufficiently large in relation to the forcing terms. This is motivated by the corresponding result for the forced and damped KdV \cite{CR} and for the Zakharov system \cite{EMT}. We show that for any $\alpha$, as long as $\gamma$ and $\delta$ are sufficiently large in relation to $\|f\|_{H^1}$ and $\|g\|_{H^1}$, the time-independent version of \eqref{eq:MBforced} has a solution in $H^1$. For values of $\alpha$ at which the system exhibits smoothing, we show that the solutions to \eqref{eq:MBforced} converge to this stationary solution in $H^1$. The proof uses a modified version of $H^1$ conservation law to obtain control over the difference between a solution and the stationary evolution.  We also prove a similar result for the $L^2$ attractor in the case $\alpha = 1$. 

\subsection{Notation} \label{notation}

The Fourier sequence of a function $u \in L^2(\mathbb{T})$ is defined by 
\[ u_k = \frac1{2\pi} \int_0^{2\pi} u(x) e^{-ikx} \; \d x \quad \text{ for } k \in \mathbb{Z}. \]
We use Sobolev spaces $H^s(\T)$, with their norms given by 
\[ \| u \|_{H^s} = \| \lb k \rb^s u_k \|_{\ell^2_k}, \]
where $\lb k \rb = (1 + |k|^2)^{1/2}$. The notation $\dot{H}^s$ indicates the mean-zero counterpart of this space, i.e. $\dot{H}^s = \left\{ u \in H^s \; | \;\int u \; \d x = 0 \right\}$. 
The estimates use the Bourgain spaces corresponding to the $u$ and $v$ evolutions. These are defined as follows: 
\begin{align*}
\|u\|_{X^{s,b}_1} &= \| \lb k \rb ^s \lb \tau - k^3 \rb ^b u_k(\tau)\|_{L^2_\tau \ell^2_k} \\
\|v\|_{X^{s,b}_\alpha} &= \| \lb k \rb ^s \lb \tau -\alpha k^3 \rb ^b v_k(\tau)\|_{L^2_\tau \ell^2_k}.
\end{align*}
We also define restricted versions of the norms:
\begin{align*}
\|u\|_{X^{s,b}_{1,\delta}} &= \inf_{u = \tilde{u}, \; |t| \leq \delta} \| \tilde{u}\|_{X_1^{s,b}} \\
\|v\|_{X^{s,b}_{\alpha,\delta}} &= \inf_{v = \tilde{v}, \; |t| \leq\delta} \| \tilde{v}\|_{X^{s,b}_\alpha}.
\end{align*}

The expressions $e^{-t\p_x^3}u_0$ and $e^{-\alpha t \p_x^3}v_0$ denote the solutions to the linear problems 
\[ 
\left\{
\begin{array}{ll}
u_t + u_{xxx} = 0 \\
u(0) = u_0
\end{array}
\right.  
\qquad \text{ and } \qquad
\left\{
\begin{array}{ll}
v_t + \alpha v_{xxx} = 0 \\
v(0) = v_0
\end{array}
\right.  
\]
respectively. We write $U(t)$ for the semigroup operator corresponding to the Majda-Biello evolution. The phase space of this operator is $\dot{H}^s \times H^s$ for $\alpha \neq 1$; when $\alpha = 1$ we work with the phase space $\dot{L}^2 \times \dot{L}^2$. 

The notation $\sum^*$ indicates summation over all terms for which the denominator of the summand is nonzero. We write $a \lesssim b$ to indicate that there is an absolute constant $C$ such that $a \leq Cb$. The symbols $\gtrsim$ is used similarly. The expression $a \approx b$ means that $a \lesssim b$ and $a \gtrsim b$. The notation $a \simeq b$ is used to indicate that $| a- b| \leq \delta$ for some small $\delta$ determined by the context. We write $a - $ for $a - \epsilon$ when $\epsilon > 0$ is arbitrary; similarly we write $a +$ for $a + \epsilon$. To simplify calculations, we use the notation $\mathcal{O}(\epsilon)$ to denote a constant of the form $C\epsilon$, where $C$ may depend on $\alpha$, but not on any of the variables in the calculation.

\section{Statement of Results}

\subsection{Background}

To study well-posedness, Oh  in \cite{Oh} used the minimal type index $\nu_\rho$, a parameter which quantifies how ``close" the number $\rho$ is to being rational. Quantities of this type are heavily studied in the theory of diophantine approximations to irrational numbers. In our case, it is important in controlling the resonances which arise in estimates. 
\begin{definition}[\cite{Arn}, \cite{Oh}]
A number $\rho \in \R$ is said to be of type $\nu$ if there exists $K > 0$ such that for all $m,n \in \mathbb{Z}$, 
\[\left| \rho - \frac{m}{n} \right| \geq \frac{K}{|n|^{2 + \nu}}. \]
The minimal type index of a number $\rho$ is defined to be 
\[\nu_\rho = \begin{cases} \infty &\rho \in \Q \\ \inf\{\nu > 0 \; | \; \rho \text{ is of type } \nu \} &\rho \notin \Q. \end{cases}\]
\end{definition}
Dirichlet's approximation theorem implies that $\nu_\rho \geq 0$ for every real number $\rho$. Furthermore, it is known that $\nu_\rho = 0$ for almost every $\rho \in \R$ \cite{Arn}. In general, though, determining the minimal type index of a specific number is difficult. In fact, it is not even known whether there is any $\rho$ such that $0 < \nu_\rho < \infty$. However, for irrational algebraic numbers we have $\nu_\rho = 0$ due to the Thue-Siegel-Roth theorem \cite{Roth}. 

The local theory depends on the minimal type index of certain parameters  $c_1$, $c_2$, $d_1$, and $d_2$ which arise in the resonance equations. The $X^{s,b}$ estimates yield resonance equations of the form $k^3 - \alpha k_1^3 - \alpha(k-k_1)^3$ and $\alpha k^3 - k_1^3 - \alpha (k-k_1)^3$. The roots of the former equation are $k_1 = c_1k$, $k_1 = c_2 k $, and $k=0$, where 
\begin{align*} 
c_1 = \frac12 + \frac{\sqrt{-3 + 12/\alpha}}{6} \qquad c_2 = \frac12 -  \frac{\sqrt{-3 + 12/\alpha}}{6}.
\end{align*}
Note that these are the roots of the quadratic $3\alpha x^2 - 3 \alpha x + \alpha - 1$, so they are algebraic for rational $\alpha$. The solutions to the second resonance equation are $k_1 = d_1k$, $k_1 = d_2k$, and $k_1 = 0$, where
$d_1 = c_1^{-1}$ and $d_2 = c_2^{-1}$. 
These are the roots of the quadratic $(1-\alpha)x^2 + 3\alpha x - 3\alpha$. 

For $\alpha$ outside $[0,4]$, the roots are not real, meaning that the resonances don't cause trouble in the estimates. In this case, the local theory is like that of the KdV. The problem for $\alpha \in (1,4]$ can be treated in the same way as that for $\alpha \in (0,1)$. For simplicity, we state results for $\alpha \in (0,1)$. 

To give the local theory precisely, define
\[ \nu_c = \nu_{c_1} = \nu_{c_2} \qquad \text{ and } \qquad \nu_d = \max\{\nu_{d_1}, \nu_{d_2}\}. \]
\begin{thm}[\cite{Oh}]
Let $\alpha \in (0,1)$. For $s \geq \min\{1, \frac12 + \frac 12 \max\{\nu_c ,\nu_d \} +\} $, the Majda-Biello initial value problem is locally 
well-posed in $\dot{H}^s \times H^s$. In particular, for any $(u_0,v_0) \in \dot{H}^s \times H^s$, there exists $T  \gtrsim (\|u_0\|_{H^s} + \|v_0\|_{H^s})^{-3}$ such that there is a unique solution $(u,v)$ to \eqref{eq:MB} satisfying
\[(u,v) \in C([-T,T]; H^s_x(\mathbb{T})) \times C([-T,T]; H^s_x(\mathbb{T})) \]
and
\[ \| u \|_{X^{s,1/2}_{1,T}} + \|v\|_{X^{s,1/2}_{\alpha,T}} \lesssim \|u_0\|_{H^s} + \|v_0\|_{H^s}. \]
\end{thm}

\subsection{Smoothing Estimate} 

The smoothing result for the nonlinear part of the Majda-Biello evolution is as follows. 
\begin{thm}\label{smoothingthrm}
Fix $\alpha \in (0,1)$ and $s > \frac12$. Consider the solution of \eqref{eq:MB} with initial data $(u_0, v_0) \in \dot{H}^s \times H^s$. 
Let
\[s^* < \min\bigg\{s - \nu_c, s - \nu_d, 2s -1 - \nu_c, 2s-1 - \nu_d \bigg\}. \]
If $\alpha = q^2/(3p(p-q) + q^2)$ for some $p, q \in \mathbb{Z}$ with $p > q$, we must instead take $s^* \leq \min\{ \frac12 - , s-1\}$. 
Then for $s_1 = s +  s^*$, we have 
\begin{align*}
u(t) - e^{-t\p^3_x}u_0 &\in C^0_tH^{s_1}_x \\
v(t) - e^{-\alpha t \p^3_x}v_0 &\in C^0_tH^{s_1}_x.
\end{align*}
In particular, for almost every $\alpha$, the above statements hold with $s_1 -  s< \min\{\frac12, s - \frac12\}$. 

If there is a growth bound 
\[ \|u(t)\|_{H^s} + \|v(t)\|_{H^s} \lesssim (1 +|t|)^{g(s)}, \]
then we also have
\[ \|u(T) - e^{-T \p_x^3}u_0\|_{H^{s_1}} + \|v(T) - e^{-\alpha T \p_x^3}v_0\|_{H^{s_1}} \leq C T^{1 + 6 g(s)}, \] 
where $C = C(s, s_1, \alpha, \|u_0\|_{H^s}, \|v_0\|_{H^s})$. 
\end{thm}

\begin{remark}
When $\alpha$ is a rational number which cannot be written in the form $q^2/(3p(p-q) + q^2)$ for some integers $p > q$, the coefficients $c_i$ and $d_i$ are irrational algebraic numbers, implying that $\nu_c = \nu_d = 0$. In this case, the best possible smoothing given by Theorem \ref{smoothingthrm} is attained. In contrast, for rationals of the form $q^2/(3p(p-q) + q^2)$, the theorem gives no smoothing unless $s>1$. For examples of such rationals, notice that no rational of the form $\ell/3^k$, where $\ell$ is not divisible by $3$, can be written as $q^2/(3p(p-q) + q^2)$. Thus these rationals form a dense subset of $[0,1]$. The rationals which \emph{are} of the form $q^2/(3p(p-q) + q^2)$ are also dense. 
\end{remark}

\begin{remark}
For $\alpha = 1$, the smoothing results for the KdV contained in \cite{ET1} can be applied to the system directly as long as we take initial data in $\dot{H}^s \times \dot{H}^s$. This implies that for any $s > -\frac12$ the nonlinear part of the evolution is in $C^0_tH^{s_1}_x$ for $s_1 \leq \min\{3s, s+1\}$.  
\end{remark}

\begin{remark}
For $s\geq 1$, well-posedness holds for \emph{any} choice of $\alpha$. However, the smoothing in the theorem above is dependent on $\alpha$, even for large $s$. It can be shown that the methods used to prove this smoothing cannot be applied to get smoothing for $\alpha$ such that $\nu_c$ or $\nu_d$ is large and finite, regardless of the size of $s$. The problem of obtaining smoothing for all $\alpha$ when $s$ is large remains open. 
The difference between the LWP results and the smoothing arises since well-posedness is proved in $X^{s,b}$ spaces, requiring estimates of multipliers such as
\[ \frac{k \lb k \rb^s}{\lb  k^3 - \alpha k_1^3 - \alpha(k-k_1)^3 \rb^{1-b} \lb k_1 \rb^s \lb k-k_1 \rb^s }. \]
For sufficiently large $s$, the estimates can be completed without a contribution from the resonant term, i.e., one can estimate $\lb k^3 - \alpha k_1^3 - \alpha(k-k_1)^3 \rb \gtrsim 1$. However, the smoothing estimates are proved using differentiation by parts, which introduces multipliers of the form 
\[ \frac{k \lb k \rb^{s_1}}{(  k^3 - \alpha k_1^3 - \alpha(k-k_1)^3 ) ^{1-b} \lb k_1 \rb^s \lb k-k_1 \rb^s }. \]
In this case, the denominator can be arbitrarily small, and the estimates cannot be completed without controlling it in some way. 
\end{remark}

Smoothing estimates can be used to obtain rough bounds on higher-order Sobolev norms by an iterative argument. Such bounds are of particular interest since the system is not completely integrable and no high regularity conservation laws exist. 

\begin{cor}\label{polybounds}
For almost every $\alpha \in (0,1)$ and for any $s \geq 1$, the global solution of \eqref{eq:MB} with $\dot{H}^s \times H^s$ initial data satisfies the growth bound 
\[ \|u(t)\|_{H^s} + \|v(t)\|_{H^s} \leq C(1 + |t|)^{\tilde{C}}, \]
where $C = C( s, \alpha, \|u_0\|_{H^s}, \|v_0\|_{H^s})$ and $\tilde{C} = \tilde{C}(s)$. 
\end{cor}
\begin{proof}
For $s = 1$, the solution is bounded in $H^1$ for all time by the conservation of the Hamiltonian. Take $\alpha$ such that $\nu_c = \nu_d  = 0$. Assume inductively that the statement of the corollary holds for some $s_0 \geq 1$. Then for $s \in (s_0, s_0 + \frac12)$ and initial data in $\dot{H}^s \times H^s$, solutions satisfy
\begin{align*}
\|u(t)\|_{H^s} + \|v(t)\|_{H^s} \leq \|u_0\|_{H^s} + \|u(t) - e^{-t \p_x^3}u_0\|_{H^{s}} + \|v_0\|_{H^s} + \|v(t) - e^{-\alpha t \p_x^3}v_0\|_{H^{s}} \\
 \leq C(\|u_0\|_{H^s}, \|v_0\|_{H^s}) + C(s, s_0, \alpha, \|u_0\|_{H^{s_0}}, \|v_0\|_{H^{s_0}})|t|^{1 + 6C(s_0)}.
\end{align*}
Repeating this argument, we can obtain the statement of the corollary for any $s \geq 1$. 
\end{proof}

Polynomial bounds for higher Sobolev norms of solutions to the KdV equation have been studied by Bourgain \cite{Bourg2} and Staffilani \cite{Staff}. Their methods use careful $X^{s,b}$ space estimates and are much more refined than the simple induction used to prove Corollary \ref{polybounds}. More recently, Kappeler, Schaad, and Topalov obtained uniform bounds for KdV solutions in all Sobolev spaces $H^s$ with $s \geq 0$ using perturbative expansions of the Fourier coefficients \cite{KST}. Their methods used Birkhoff normal forms, relying on the integrability of the KdV.  


\subsection{Existence of a Global Attractor} 

We use smoothing estimates for the dissipative version of the Majda-Biello system to derive the existence of a global attractor. In the following, $U(t)$ will denote the evolution operator corresponding to \eqref{eq:MBforced}. Note that the notion of a global attractor is only reasonable when the system is globally well-posed. For the forced and weakly damped system, global well-posedness holds by the restricted norm argument of Bourgain using the estimates established in \cite{Oh}; see Section $2$ of \cite{ET3} for a similar argument. We begin with the definition of a global attractor. 

\begin{definition}[\cite{Tem}] \label{ga} A compact subset $\mathcal{A}$ of the phase space $H$ is called a global attractor for the semigroup $\{U(t)\}_{t \geq 0}$ if $\mathcal{A}$ is invariant under the flow of $U$ and 
\begin{align*}
\lim_{t \to \infty} d(U(t)u_0 ,\mathcal{A}) = 0 \text{  for every  } u_0 \in H.
\end{align*}
\end{definition}

Using energy estimates, we show that all solutions eventually enter a bounded subset of $\dot{H}^1 \times H^1$. Such a set is called an absorbing set for the evolution $U(t)$:

\begin{definition}[\cite{Tem}] \label{abset} A bounded subset $\mathcal{B}_0$ of $H$ is called absorbing if for any bounded $\mathcal{B} \subset H$, there exists a time $T = T(\mathcal{B})$ such that $U(t)\mathcal{B} \subset \mathcal{B}_0$ for all $t \geq T$. 
\end{definition}

Our global attractor will be the $\omega$-limit set of $\mathcal{B}_0$, which is defined by 
\[ \omega(\mathcal{B}_0)  = \bigcap_{s \geq 0} \overline{\bigcup_{t \geq  s} U(t)\mathcal{B}_0}. \]
Notice that it is immediate that the existence of a global attractor implies the existence of an absorbing set. The converse does not hold, though; an absorbing set may not be invariant under the flow and need not be compact. A partial converse is true, however, and will be used to show that the $\omega$-limit set is indeed a global attractor. 

\begin{thm}[\cite{Tem}] \label{asymthrm}
Let $H$ be a metric space and $U(t)$ be a continuous semigroup from $H$ to itself for all $t \geq 0$. Assume that there is an absorbing set $\mathcal{B}_0$. If the semigroup $\{U(t)\}_{t \geq 0}$ is asymptotically compact, i.e. for every bounded sequence $\{x_k\} \subset H$ and every sequence $t_k \to \infty$, the set $\{U(t_k)x_k\}_k$ is relatively compact in $H$, then $\omega (\mathcal{B}_0)$ is a global attractor. 
\end{thm}

We will prove asymptotic compactness using a smoothing estimate for the dissipative system, yielding the following result. 

\begin{thm} \label{GAthrm}
For almost every $\alpha \in (0,1)$, the dissipative Majda-Biello system \eqref{eq:MBforced} has a global attractor in $\dot{H}^1 \times H^1$. Moreover, the global attractor is a compact subset of $H^{s} \times H^{s}$ for any $s < \frac 32$.   
\end{thm}

\begin{remark}
For $\alpha = 1$ and forcing terms in $\dot{L}^2$, the arguments in \cite{ET3} immediately yield a global attractor in $\dot{L}^2 \times \dot{L}^2$ which is compact in $H^s \times H^s$ for any $s < 1$.
\end{remark}
When the damping terms are sufficiently large in relation to the forcing, the attractor is trivial:
\begin{thm}
Assume $\min\{\gamma,\delta\} \geq \frac{\sqrt{\alpha C^3}}{8}$, where $C$ the norm of the embedding $H^1 \hookrightarrow L^\infty$. If $\|f\|_{H^1}, \|g\|_{H^1}\ll \big(\min\{\gamma,\delta\}\big)^{4/3}$, the global attractor given by Theorem \ref{GAthrm} consists of a single pair of functions $(p,q) \in H^5(\T) \times H^5(\T)$.  
\end{thm}

\begin{remark}
When $\alpha=1$ and the forcing terms $f$ and $g$ are in $\dot{L}^2$, the statement of the theorem holds if $\|f\|_{L^2}, \|g\|_{L^2} \ll \min\{\gamma,\delta\}$. 
\end{remark}

This theorem is proved using a modification of the Hamiltonian conservation law to show that a solution to \eqref{eq:MBforced} converges to the solution of the corresponding time-independent system as $t \to \infty$. 

\section{Proof of Theorem \ref{smoothingthrm}}

To prove the smoothing estimate, we begin by establishing an equivalent formulation of \eqref{eq:MB} via differentiation by parts. This formulation decomposes the equation into several terms which will be estimated separately. 

\begin{prop} \label{dbp} Assume $u_0 \in \dot{H}^s$. The system \eqref{eq:MB} can be written in the following form:
\[
\begin{cases}
\p_t\big[e^{-ik^3t}(u_k + B_1(v,v)_k)\big] &= e^{-ik^3t}\big[\rho_1(v,v)_k + R_1(u,v,v)_k\big] \\
\p_t\big[e^{-i\alpha k^3 t}(v_k + B_2(u,v)_k)\big] &= e^{-i\alpha k^3 t}\big[\rho_2(u,v)_k + R_2(v,v,v)_k + R_3(u,u,v)_k\big], 
\end{cases}
\]
where
\begin{align*}
B_1(u,v)_k = -\frac{k}{2} &\sum_{k_1+k_2 = k}^* \frac{u_{k_1}v_{k_2}}{k^3-\alpha k_1^3 - \alpha k_2^3}\\
B_2(u,v)_k = -k &\sum_{k_1 + k_2 = k}^* \frac{u_{k_1}v_{k_2}}{\alpha k^3 - k_1^3 - \alpha k_2^3} \\
R_1(u,v,w)_k = - \frac{i}{3\alpha} &\sum_{k_1 + k_2 + k_3 = k}^* \frac{(k_1 + k_2)u_{k_1}v_{k_2}w_{k_3}}{(k_1 + k_2 - c_1k)(k_1+k_2 - c_2 k )} \\
R_2(u,v,w)_k = \frac{ik}{2} &\sum_{k_1 + k_2 + k_3 = k}^* \frac{(k_1 + k_2) u_{k_1}v_{k_2}w_{k_3}}{\alpha k^3 - (k_1 + k_2)^3 - \alpha k_3^3} \\
R_3(u,v,w)_k = ik &\sum_{k_1 + k_2 + k_3 = k}^* \frac{(k_2 + k_3)u_{k_1}v_{k_2}w_{k_3}}{\alpha k ^3 - k_1^3 - \alpha (k_2+k_3)^3} \\
\rho_1(u,v)_k = -ik&\big(u_{c_1 k }v_{c_2 k}\big) \\
\rho_2(u,v)_k = -ik&\big(u_{d_1k}v_{(1-d_1)k} + u_{d_2 k } v_{(1-d_2)k}\big).
\end{align*}
\end{prop}

\begin{proof} 
Taking the spatial Fourier transform of \eqref{eq:MB} yields
\[
\begin{cases}
\displaystyle{\partial_t u_k - i k^3 u_k + \frac{ik}2 \sum_{k_1 + k_2 = k} v_{k_1}v_{k_2} = 0} \\
\displaystyle{\partial_t v_k - i \alpha k^3 v_k + ik \sum_{k_1 + k_2 = k} u_{k_1}v_{k_2} = 0}.
\end{cases} 
\]

Change variables by setting $m_k(t)= e^{-ik^3t}u_k(t)$ and $n_k(t) = e^{-i\alpha k^3}v_k(t)$. Then the system becomes
\[
\begin{cases}
\displaystyle{ \partial_t m_k = \frac{-ik}{2} \sum_{k_1 + k_2 = k} e^{-it(k^3-\alpha k_1^3 -
\alpha k_2^3)}n_{k_1}n_{k_2} }\\
\displaystyle{\partial_t n_k = -ik \sum_{k_1 + k_2 = k}  e^{-it(\alpha k^3 - k_1^3 - \alpha
k_2^3)}m_{k_1}n_{k_2}}.
\end{cases}
\]

Differentiate the equation for $\p_t m_k$ by parts to obtain
\begin{align*}  
\partial_t m_k = \frac{k}{2} &\sum_{k_1 + k_2 = k}^* \frac{\partial_t(e^{-it(k^3-\alpha k_1^3 - \alpha
k_2^3)}n_{k_1}n_{k_2})}{k^3-\alpha k_1^3 - \alpha k_2^3} \\ 
- \frac{k}{2} &\sum_{k_1+ k_2 = k} ^* \frac{e^{-it(k^3 - \alpha k_1^3 - \alpha
k_2^3)}\p_t(n_{k_1}n_{k_2})}{k^3 - \alpha k_1^3 - \alpha k_2^3}  
- ikn_{c_1 k}n_{c_2 k}.
\end{align*}
Recall that the constants $c_1$ and $c_2$ arise from the solving $k^3 - \alpha k_1^3 - \alpha k_2^3 = 0$. The $k = 0$ solution does not appear in the resonant term since we assume that $u_0$, and hence $u$, is mean zero. Furthermore, the resonant term only appears when $c_1$ and  $c_2$ are rational and $c_1 k$ is an integer. In particular, $c_1, c_2 \in \Q$ only if $\alpha = q^2/(3p(p-q) + q^2)$ for some $p, q \in \mathbb{Z}$ with $p > q$.

Using the differential equation, we find that the second sum in $\p_tm_k$  is 
\begin{align*}
ik &\sum_{k_1+k_2+k_3 = k}^* (k_1 + k_2)\frac{e^{-it(k^3-k_1^3-\alpha k_2^3
-\alpha k_3^3)} m_{k_1}n_{k_2}n_{k_3}}{k^3-\alpha(k_1+k_2)^3 - \alpha k_3^3} \\
= -\frac{i}{3\alpha} &\sum_{k_1 + k_2 + k_3 = k}^* (k_1 + k_2)\frac{e^{-it(k^3-k_1^3-\alpha k_2^3
-\alpha k_3^3)} m_{k_1}n_{k_2}n_{k_3}}{(k_1+k_2 - c_1k)(k_1+k_2-c_2 k)}.
\end{align*}

Moving to the equation for $\p_tn$, differentiate by parts again to find
\begin{align*} 
\p_t n_k = k \sum_{k_1 +k_2 = k}^* \frac{\p _t (e^{-it(\alpha k^3 - k_1^3 -
\alpha k_2^3)} m_{k_1}n_{k_2})}{\alpha k^3 - k_1^3 - \alpha k _2^3} 
 - k \sum_{k_1 + k_2 = k}^* \frac{ \p_t(m_{k_1}n_{k_2}) e^{-it(\alpha k^3 - k_1^3 -
\alpha k_2^3)}}{\alpha k^3 - k_1 ^3 - \alpha k_2^3} \\
 - ik m_{d_1 k}n_{(1-d_1)k} - ik m_{d_2 k}n_{(1-d_2) k}.
\end{align*}
Here again, the last terms in the equality only appear when $d_1$ and $d_2$ are rational and $d_1k, d_2k \in \mathbb{Z}$. Using the differential equation, rewrite the second sum in $\p_tn_k$ as
\begin{align*}
\frac{ik}{2} &\sum_{k_1 + k_2 +k_3 = k}^* e^{-it(\alpha k^3 - \alpha k_1^3 -\alpha k_2^3 - \alpha k_3^3)} \frac{(k_1 + k_2)n_{k_1}n_{k_2}n_{k_3}}{\alpha k^3 - (k_1 + k_2)^3 - \alpha k_3^3} \\
+ ik &\sum_{k_1 + k_2 + k_3 = k}^* e^{-it(\alpha k^3 - k_1^3 - k_2^3 -\alpha k_3^3)}\frac{(k_2+k_3)m_{k_1}m_{k_2}n_{k_3}}{\alpha k^3 - k_1^3 - \alpha(k_2+k_3)^3}.
\end{align*}
Collecting all these terms and returning to $u$ and $v$ variables gives the statement of the proposition. 
\end{proof} 

We use the transformed system to get bounds on the norm of the difference between the linear and nonlinear evolution. First, integrate the new system from $0$ to $t$ to obtain
\begin{align*}
\begin{cases}
u_k(t) - e^{i k^3 t}u_k(0) = &-B_1(v,v)_k(t) + e^{ik^3t}B_1(v,v)_k(0)  \\
&+ \int_0^t e^{ik^3(t-r)} \left[ R_1(u,v,v)_k(r) + \rho_1(v,v)_k(r)\right]\;\d r \\ 
v_k(t) - e^{i\alpha k^3 t}v_k(0) = &-B_2(u,v)_k(t) + e^{i\alpha k^3t}B_2(u,v)_k(0) \\
&+ \int_0^t e^{i\alpha k^3(t-r)} \left[ R_2(v,v,v)_k(r) + R_3(u,u,v)_k(r) + \rho_2(u,v)_k(r)\right]\;\d r.  
\end{cases}
\end{align*}
To control these expressions, we use the following estimates. Propositions~\ref{B1} and \ref{B2} are proved in Section~\ref{estimateproofs}; Proposition~\ref{Rho} is immediate from the definitions of $\rho_1$ and $\rho_2$.

\begin{prop} \label{B1}
If $s > \frac12$ and $s_1 - s < \min\{ 1, s - \nu_{c} \}$, then 
\[ \| B_1(u,v) \|_{H^{s_1}_x} \lesssim \|u\|_{H_x^s}\| v \|_{H_x^s}. \]
When $\alpha = q^2/(3p(p-q) + q^2)$ for some $p, q \in \mathbb{Z}$ with $p > q$, we only require that $s_1 - s \leq 1$. 
\end{prop}

\begin{prop} \label{B2}
Let $u \in \dot{H}^s$. If $s > \frac12$ and $s_1 -s < \min\{ 1,  s - \nu_d\}$, then
\[ \|B_2(u,v)\|_{H^{s_1}_x} \lesssim \|u\|_{H^s_x} \|v\|_{H^s_x}. \]
When $\alpha = q^2/(3p(p-q) + q^2)$ for some $p, q \in \mathbb{Z}$ with $p > q$, we only need $s_1 - s \leq \min\{1,s - \}$. 
\end{prop}

\begin{prop} \label{Rho}
If $s_1 - s \leq s -1$, then
\[ \| \rho_1(u,v)\|_{H^{s_1}_x} \lesssim \| u\|_{H^s_x}\| v\|_{H^s_x} \qquad \text{and} \qquad   \| \rho_2(u,v)\|_{H^{s_1}_x} \lesssim \| u\|_{H^s_x} \| v\|_{H^s_x}. \]
\end{prop}

Using Propositions \ref{B1}-\ref{Rho} on the equations found above, write, for $s_1 -s $ sufficiently small, 
\begin{align*}
\| u(t) - e^{-t\p^3_x}u_0 \|_{H^{s_1}_x} &\lesssim \; \| v(t) \|_{H^s_x}^2 + \| v(0)  \|_{H^s_x}^2 + \int_0^t \|v(r)\|_{H^s_x}^2 \;\d r  \\
&+ \; \left\|\int_0^t e^{-(t-r)\p^3_x}  R_1(u,v,v)(r)\;\d r \right\|_{H^{s_1}_x} \\
\| v(t) - e^{-\alpha t\p^3_x}v_0 \|_{H^{s_1}_x} &\lesssim \; \| u(t) \|_{H^s_x} \| v(t)\|_{H^s_x} + \|u(0) \|_{H^s_x}\|v(0)\|_{H^s_x}  + \int_0^t \|u(r)\|_{H^s_x}\|v(r)\|_{H^s_x} \;\d r  \\
&+ \; \left\|\int_0^t e^{-\alpha(t-r)\p^3_x}[R_2(v,v,v)(r) + R_3(u,u,v)(r)]\;\d r \right\|_{H^{s_1}_x}. 
\end{align*}

To complete the estimates, we need the following bounds for $R_1$, $R_2$, and $R_3$. See Section~\ref{estimateproofs} for proofs. 

\begin{prop} \label{R1}
Let $u \in \dot{H}^s$. For $b - \frac 12$ sufficiently small, $s > \frac12$, and $s_1 - s < \min\{1, 2s - 1 - \nu_{c_1}, s - \frac12 , s + \frac12 - \nu_{c_1}\}$, we have
\[  \|R_1(u,v,w) \|_{X^{s_1,b-1}_1} \lesssim \|u\|_{X^{s,1/2}_1} \|v\|_{X^{s,1/2}_\alpha}\|w\|_{X^{s,1/2}_\alpha}. \]
When $\alpha = q^2/(3p(p-q) + q^2)$ for $p, q \in \mathbb{Z}$ with $p > q$, we only need $s_1 - s \leq \min\{1,s - \frac12 \}$. 
\end{prop}

\begin{prop} \label{R2}
For $b - \frac12$ sufficiently small, $s > \frac12$, and $s_1 - s < \min\{2s - 1 - \nu_d, s + \frac12 - \nu_d \}$,
\[ \| R_2(u,v,w)\|_{X^{s_1, b-1}_\alpha} \lesssim  \| u \|_{X^{s,1/2}_\alpha}\| v \|_{X^{s,1/2}_\alpha} \| w \|_{X^{s,1/2}_\alpha}. \]
When $\alpha = q^2/(3p(p-q) + q^2)$ for $p, q \in \mathbb{Z}$ with $p > q$, we only need $s_1 - s \leq 1$.
\end{prop}

\begin{prop} \label{R3} Let $u \in \dot{H}^s$. 
For $b - \frac12$ sufficiently small, $s > \frac12$, and $s_1 - s < \min\{\frac12, s -\frac12, 2s - 1 - \nu_d, s + \frac12 - \nu_d \}$,
\[ \| R_3(u,u,v)\|_{X^{s_1, b-1}_\alpha} \lesssim  \| u \|_{X^{s, 1/2}_1}^2 \| v \|_{X^{s,1/2}_\alpha}. \]
When $\alpha = q^2/(3p(p-q) + q^2)$ for $p, q \in \mathbb{Z}$ with $p > q$, we only need $s_1 - s \leq \min\{\frac12 -,s - \frac12 \}$.
\end{prop}

We will use these estimates, the embeddings $X^{s_1,b}_1, \; X^{s_1,b}_\alpha \hookrightarrow L^\infty_tH^{s_1}_x$ for $b > \frac12$, and the following standard lemma to complete the proof. Here $\eta$ is a smooth function supported on $[-2,2]$ with $\eta = 1$ on $[-1,1]$, and $\eta_\delta$ is defined by $\eta_\delta(t) = \eta(t/\delta)$.  

\begin{lem}[\cite{GTV}]
For $b \in (\frac12, 1]$, we have
\begin{align*}
\left\| \eta(t) \int_0^t e^{-(t-r)\p^3_x}F(r) \;\d r \right\|_{X^{s_1,b}_1} &\lesssim  \|F\|_{X^{s_1,b-1}_{1,\delta}} \\
\left\| \eta(t) \int_0^t e^{-\alpha(t-r)\p^3_x}F(r) \;\d r \right\|_{X^{s_1,b}_\alpha} &\lesssim \|F\|_{X^{s_1,b-1}_{\alpha,\delta}}. 
\end{align*}
\end{lem}

Let $\delta$ be the existence time for the system from the local theory. Then for $t \in [-\delta/2, \delta/2]$, we have 
\begin{align*}
\left\|\int_0^t e^{-(t-r)\p^3_x}  R_1(u,v,v)(r)\;\d r \right\|_{H^{s_1}_x} 
\leq \; \left\| \eta_\delta(t)\int_0^t e^{-(t-r)\p^3_x}  R_1(u,v,v)(r)\;\d r \right\|_{L^\infty_tH^{s_1}_x} \\
\lesssim \;  \left\| \eta_\delta(t)\int_0^t e^{-(t-r)\p^3_x}  R_1(u,v,v)(r)\;\d r \right\|_{X^{s_1,b}_1} 
\lesssim \;  \| R_1(u,v,v) \|_{X^{s_1,b-1}_{1,\delta}} 
\lesssim \;  \|u\|_{X^{s,1/2}_{1,\delta}}\|v\|^2_{X^{s,1/2}_{\alpha,\delta}}. 
\end{align*}

Similarly, from the second equation we find
\begin{align*}
&\left\|\int_0^t e^{-\alpha(t-r)\p^3_x}[R_2(v,v,v)(r) + R_3(u,u,v)(r)]\;\d r \right\|_{H^{s_1}_x} \\
\leq \; &\left\| \eta_\delta(t) \int_0^t e^{-\alpha(t-r)\p^3_x}[R_2(,v,v,v)(r) + R_3(u,u,v)(r)]\;\d r \right\|_{L^\infty_tH^{s_1}_x} \\
\lesssim \;  &\left\| \eta_\delta(t) \int_0^t e^{-\alpha(t-r)\p^3_x}[R_2(v,v,v)(r) + R_3(u,u,v)(r)]\;\d r \right\|_{X^{s_1,b}_\alpha} \\
\lesssim \; & \|R_2(v,v,v) \|_{X^{s_1,b-1}_{\alpha,\delta}} + \|R_3(u,u,v) \|_{X^{s_1,b-1}_{\alpha,\delta}} 
\lesssim \; \| v\|^3_{X^{s,1/2}_{\alpha,\delta}} + \| u\|^2_{X^{s,1/2}_{1,\delta}}\| v\|_{X^{s,1/2}_{\alpha,\delta}} . 
\end{align*}
Thus, collecting these estimates, we have
\begin{align*}
\| u(t) - e^{-t\p^3_x}u_0 \|_{H^{s_1}_x} &\lesssim \; \| v(t) \|_{H^s_x}^2 + \| v(0)  \|_{H^s_x}^2 + \int_0^t \|v(r)\|_{H^s_x}^2 \;\d r  
+ \|u\|_{X^{s,1/2}_{1,\delta}}\|v\|^2_{X^{s,1/2}_{\alpha,\delta}} \\
\| v(t) - e^{-\alpha t\p^3_x}v_0 \|_{H^{s_1}_x} &\lesssim \; \| u(t) \|_{H^s_x} \| v(t)\|_{H^s_x} + \|u(0) \|_{H^s_x}\|v(0)\|_{H^s_x}  + \int_0^t \|u(r)\|_{H^s_x}\|v(r)\|_{H^s_x} \;\d r  \\
&+\| v\|^3_{X^{s,1/2}_{\alpha,\delta}} + \| u\|^2_{X^{s,1/2}_{1,\delta}}\| v\|_{X^{s,1/2}_{\alpha,\delta}}. 
\end{align*}
Combining the estimates for the two equations, we may write
\begin{align*}
&\| u(t) - e^{-t\p^3_x}u_0 \|_{H^{s_1}_x} + \| v(t) - e^{-\alpha t\p^3_x}v_0 \|_{H^{s_1}_x}  \lesssim  \; \big(\|u(0)\|_{H^s_x} + \| v(0) \|_{H^s_x} \big)^2 \\
&+\big(\|u(t)\|_{H^s_x} + \| v(t) \|_{H^s_x} \big)^2 
 + \int_0^t \Big(\|u(r)\|_{H^s_x} + \| v(r) \|_{H^s_x} \Big)^2 \;\d r  + \Big(\| u\|_{X^{s,1/2}_{1,\delta}} + \| v\|_{X^{s,1/2}_{\alpha,\delta}} \Big)^3.
\end{align*}
 
We demonstrate the polynomial growth bound and then the continuity. Fix $T$ large. For $t \leq T$, we have the bound 
 \[ \|u(t)\|_{H^s_x} + \|v(t)\|_{H^s_x} \lesssim (1 + |t|)^{g(s)} \lesssim T^{g(s)}. \]
 Then for $\delta \sim T^{-3 g(s) }$ and any $j \leq T/\delta \sim T ^{1+3 g(s)}$, we have 
 \[ \|u(j \delta) - e^{-t\p^3_x}u((j-1)\delta) \|_{H^{s_1}_x} +  \|v(j \delta) - e^{-\alpha t \p^3_x } v((j-1)\delta) \|_{H^{s_1}_x} \lesssim T^{3g(s)}, \]
 using the local theory bound
 \[ \|u\|_{X^{s, 1/2}_{1, [(j-1)\delta, j \delta]} } + \|v\|_{X^{s, 1/2}_{\alpha, [(j-1)\delta, j \delta]} } \lesssim \|u((j-1)\delta)\|_{H^s} + \|v((j-1)\delta)\|_{H^s} \lesssim T^{g(s)}. \]
Now let $J = T/\delta \sim T^{1 + 3g(s)}$ and write
\begin{align*}
\|u(J\delta) - e^{-J\delta \p_x^3}u_0\|_{H^{s_1}} &\leq \sum_{j=1}^J \|u(j \delta) - e^{-t\p^3_x}u((j-1)\delta) \|_{H^{s_1}_x}  \lesssim T^{1 + 6g(s)}. 
\end{align*}
The corresponding estimate for $v$ completes the proof of the growth bound. 

To prove continuity, write
\begin{align*}
u_k(t) - u_k(t') &= (e^{ik^3t} - e^{ik^3t'})\big[u_k(0) + B_1(v,v)_k(0)\big] 
+ B_1(v,v)_k(t') - B_1(v,v)_k(t) \\
&+ \int_0^t e^{ik^3(t-r)} R_1(u,v,v)_k(r) \;\d r - \int_0^{t'} e^{ik^3(t'-r)} R_1(u,v,v)_k(r) \;\d r  \\
&+ \int_0^t e^{ik^3(t-r)} \rho_1(v,v)_k(r) \;\d r  - \int_0^{t'} e^{ik^3(t'-r)} \rho_1(v,v)_k(r) \;\d r.
\end{align*}
The continuity follows by applying the estimates stated previously along with the continuity of $u$ in $H^s$; see \cite{ET1}. Continuity of $v$ is proved in the same way.

\section{Proof of Existence of Global Attractors}

We will consider the forced and damped version \eqref{eq:MBforced} of the Majda-Biello system with $\gamma, \delta > 0$. For simplicity, take $\gamma = \delta$; minor modifications to the calculations extend them to the general case. The first step is to obtain bounds on the $H^1$ norms for the dissipative system. This will imply the existence of an absorbing set (see Definition \ref{abset}). Recall the conservation laws \eqref{eq:conservationlaws} for the original Majda-Biello system.
To get a bound in the dissipative case, we study $E_3$ and $E_4$ in the presence of dissipation.
\begin{lem} \label{H1bound}
Solutions to \eqref{eq:MBforced} satisfy
\begin{align*}
\|u(t)\|_{H^1} + \|v(t)\|_{H^1} \leq C = C(\|u_0\|_{H^1}, \|v_0\|_{H^1}, \|f\|_{H^1} + \|g\|_{H^1}, \gamma, \alpha). 
\end{align*}
\end{lem}

\begin{proof} In the following manipulations, $C$ and $\tilde{C}$ are positive constants whose value may change from one side of an inequality to the other.  Recall that $E_3 = \int u^2 + v^2 \; \d x$. Then  
\begin{align*}
\p_t E_3 + 2 \gamma E_3 =  2\int fu + gv \; \d x \leq 2\| f\|_{L^2}\|u\|_{L^2} + 2\|g\|_{L^2}\|v\|_{L^2} \leq 4\big(\|f\|_{L^2}+ \|g\|_{L^2}\big)\sqrt{E_3}.
\end{align*}
Let $F_3 = e^{2\gamma t} E_3$. The above inequality gives $\p_t F_3 \leq 4e^{\gamma t} \big(\|f\|_{L^2} + \|g\|_{L^2} \big) \sqrt{F_3}$, or
\[ \p_t \sqrt{F_3} \leq 2e^{\gamma t} \big(\|f\|_{L^2} + \|g\|_{L^2}\big). \] 
Integrating this inequality and rewriting $F_3$ in terms of $u$ and $v$ norms gives
\begin{align*}
 \sqrt{\|u(t)\|_{L^2}^2 + \|v(t)\|_{L^2}^2 } &\leq e^{-\gamma t} \sqrt{\|u_0\|_{L^2}^2 + \|v_0\|_{L^2}^2} + 2\frac{\|f \|_{L^2} + \|g\|_{L^2}}{\gamma} (1 - e^{-\gamma t})  \\
 &\leq C = C(\|u_0\|_{L^2}, \|v_0\|_{L^2}, (\|f\|_{L^2} + \|g\|_{L^2}), \gamma). 
 \end{align*}
Thus the $L^2$ norms of the $u$ and $v$ are bounded in the dissipative case. Next consider $E_4 = \int u_x^2 + \alpha v_x^2 - uv^2 \; \d x$. First notice that $E_4$ is bounded below due to the bound on $\| v\|_{L^2}$ and the embedding $H^1 \hookrightarrow L^\infty$. To get an upper bound, use the embedding again to write
\begin{align*} 
\p_t E_4 + &2 \gamma E_4 = 2 \int f_x u_x + g_xv_x  \; \d x - \int  fv^2 + guv \; \d x + \gamma \int uv^2 \; \d x \\
& \lesssim \|f\|_{H^1} \|u_x\|_{L^2} + \|g\|_{H^1} \|v_x\|_{L^2} + \|f\|_{H^1} \|v\|_{L^2}^2  + \|g\|_{H^1}\|u\|_{L^2} \|v\|_{L^2} + \|u\|_{H^1} \|v\|_{L^2}^2   \\
&\leq C + \tilde{C}\big( \|u_x\|_{L^2} + \|v_x\|_{L^2}\big). 
\end{align*}
The constants in second inequality depend on the bounds on $\|u\|_{L^2}$ and $\|v\|_{L^2}$ and on the value of $\|f\|_{H^1}$. Now note that 
\begin{align*} 
\|u_x\|_{L^2}^2 + \alpha \|v_x\|_{L^2}^2  = E_4+ \int uv^2 \; \d x  &\leq E_4 + C\| v\|_{L^2}^2 \|u \|_{H^1} \\
&\leq E_4 + C(\|u_x\|_{L^2}^2 + C)^{1/2} 
\leq (E_4 + C) + \tilde{C}\|u_x\|_2. 
\end{align*}
The second inequality uses the $L^2$ bounds on $u$ and $v$. Then we have 
\[ \|v_x\|_{L^2} \lesssim \sqrt{\big(\|u_x\|_{L^2} - {\tilde{C}}/{2}\big)^2 + \alpha \|v_x\|_{L^2}^2} \leq \sqrt{ E_4 + C + {\tilde{C}^2}/{4}} \lesssim \sqrt{|E_4|} + C, \]
and similarly
\[ \|u_x\|_{L^2} - \tilde{C}/2 \lesssim \sqrt{\big(\|u_x\|_{L^2} - {\tilde{C}}/{2}\big)^2 + \alpha \|v_x\|_{L^2}^2} \lesssim \sqrt{|E_4|} + C. \]
This implies that $\|u_x\|_{L^2} + \|v_x\|_{L^2} \lesssim \sqrt{|E_4|} + C$. Using this bound with the change of variables $F_4 = e^{2\gamma t} E_4$, we have 
\begin{align*}
\p_t F_4 \leq e^{\gamma t} \left[ Ce^{\gamma t} + \tilde{C}\sqrt{|F_4|}\right]. 
\end{align*}
Then 
\begin{align*}
E_4(t) &\leq e^{-2\gamma t}E_4(0) + C \frac{1 - e^{-2\gamma t}}{2 \gamma} + \tilde{C} \int_0^t e^{-2 \gamma (t-s)} \sqrt{|E_4(s)|} \; ds \\
&\leq  E_4(0) + C + \tilde{C} \| \sqrt{E_4} \|_{L^\infty([0,t])}.
\end{align*}

Now take $M \gg 1$, and suppose $E_4$ attains the value $M$. Let $t$ be the first time the value is attained. Then  $M \leq C + \tilde{C} \sqrt{M}$,
which is impossible for sufficiently large $M$. Thus $E_4$ is bounded above. 
\end{proof}

With this lemma, we conclude that solutions of the dissipative Majda-Biello system remain in a ball, say $\mathcal{B}_0$, in the space $H^1 \times H^1$. We now show that the $\omega$-limit set of the ball, 
\[ \omega(\mathcal{B}_0)  = \bigcap_{s \geq 0} \overline{\bigcup_{t \geq  s} U(t)\mathcal{B}_0}, \]
is a global attractor in the sense of Definition \ref{ga}. Lemma \ref{H1bound} gives the existence of an absorbing set for \eqref{eq:MBforced}, so by Theorem \ref{asymthrm} we only need prove asymptotic compactness of $U(t)$. To do so, we use the following general smoothing estimate. Notice that it gives a bound on the nonlinear evolution minus a correction involving the resonant terms $\rho_i$. In Theorem~\ref{GAthrm}, we consider only the full-measure set of $\alpha$ such that $\rho_1 = \rho_2 = 0$. In this situation, the correction terms vanish. 

\begin{thm} \label{MBforcedsmoothing}
Consider the solution of \eqref{eq:MBforced} with initial data $(u_0, v_0) \in \dot{H}^1 \times H^1$. Then for any $a < \min\{1-\nu_c, 1 - \nu_d\}$, we have
\begin{align*}
\left\| u(t) - e^{-t\p_x^3 - \gamma t} u_0  - \int_0^t  e^{(-\p_x^3 - \gamma )(t-r)} \rho_1(v,v)(r)\d r \right\|_{H^{1+a}} \\
+ \left\| v(t) - e^{-\alpha t\p_x^3 - \gamma t} v_0  - \int_0^t  e^{(-\alpha\p_x^3 - \gamma )(t-r)}\rho_2(u,v)(r)\d r \right\|_{H^{1+a}} \\
\leq C(a, \gamma, \|f\|_{H^1}, \|g\|_{H^1}, \|u_0\|_{H^1}, \|v_0\|_{H^1}). 
\end{align*}
\end{thm}

\begin{proof}
Taking the  Fourier transform of \eqref{eq:MBforced} yields
\[
\begin{cases}
\displaystyle{\partial_t u_k - i k^3 u_k + \gamma  u_k+ \frac{ik}2 \sum_{k_1 + k_2 = k} v_{k_1}v_{k_2} = f_k} \\
\displaystyle{\partial_t v_k - i \alpha k^3 v_k +  \gamma  v_k + ik \sum_{k_1 + k_2 = k} u_{k_1}v_{k_2} = g_k}.
\end{cases} 
\]

Change variables by setting $m_k = e^{-ik^3t+\gamma t} u_k$ and             $n_k = e^{-i\alpha k^3 t + \gamma t} v_k$, with $p_k(t) = e^{-ik^3t+\gamma t} f_k$ and $q_k(t) = e^{-i\alpha k^3 t + \gamma t} g_k$. After the change of variables, the system is
\[
\begin{cases}
\displaystyle{\partial_t m_k = -  \frac{ik}2 \sum_{k_1 + k_2 = k} e^{-it(k^3 - \alpha k_1^3 - \alpha k_2^3)}n_{k_1}n_{k_2}}+ p_k \\
\displaystyle{\partial_t n_k =  -  ik \sum_{k_1 + k_2 = k} e^{-it(\alpha k^3 - k_1^3 - \alpha k_2^3)}m_{k_1}n_{k_2}} + q_k.
\end{cases} 
\]

Then differentiating by parts as before gives the equivalent formulation
\[
\begin{cases}
\displaystyle{\partial_t\Big[e^{-ik^3t+\gamma t}\big(u_k + B_1(v,v)_k\big)\Big] = e^{-ik^3t+\gamma t}\Big[ \rho_1(v,v)_k + R_1(u,v)_k + B_1(g,v)_k+ f_k \Big]}\\
\displaystyle{\partial_t\Big[e^{-i\alpha k^3 t + \gamma  t}\big(v_k + B_2(u,v)_k\big)\Big]  } = \\
\qquad {e^{-i\alpha k^3 t +  \gamma  t}\Big[\rho_2(u,v)_k +  R_2(v,v,v)_k + R_3(u,u,v)_k +B_2(f,v)_k + B_2(g,u)_k+ g_k\Big]},
\end{cases} 
\]
where $\rho_j$, $B_j$, and $R_j$ are defined as in Proposition~\ref{dbp}. Integrating from $0$ to $t$ yields the equations
\begin{align*}
u_k(t) &- e^{ik^3t - \gamma t}u_k(0) 
= -B_1(v,v)_k + e^{ik^3 t -  \gamma  t}B_1(v_0,v_0) \\
 &+ \int_0^t e^{(ik^3 -  \gamma )(t-s)}\Big[\rho_1(v,v)_k + f_k  + R_1(u,v)_k + B_1(v,g)_k\Big] \d s \\
v_k(t) &- e^{-\alpha k^3 t - \gamma t}v_k(0) 
 = - B_2(u,v)_k(t) + e^{i\alpha k^3 t- \gamma t} B_2(u_0,v_0) \\
+ \int_0^t &e^{ (i\alpha k^3 - \gamma )(t-s)}\Big[\rho_2(u,v)_k + B_2(f,v)_k + B_2(u,g)_{k} + R_2(v,v,v)_k + R_3(u,u,v)_k + g_k\Big] \d s. 
\end{align*}
Note that 
\[ \left \| \int_0^ t e^{(ik^3 - \gamma)(t-s)}f_k \; \d s \right\|_{H^{s+a}}  = \left \| \frac{e^{(ik^3 - \gamma)t} - 1}{ik^3 - \gamma} f_k \right\|_{H^{s+a}} \lesssim \|f\|_{H^{s-2}}. \]
This, the corresponding estimate for $e^{(i\alpha k^3 - \gamma)(t-s)}g_k$, the estimates used for the previous smoothing result, and Lemma \ref{H1bound} give the following estimates for $t < \delta$, where $\delta$ is the existence time from the local theory:  
\begin{align*}
\left\| u(t) - e^{-t\p_x^3 - \gamma t} u_0  - \int_0^t  e^{(-\p_x^3 - \gamma )(t-r)} \rho_1(v,v)(r)\d r \right\|_{H^{1+a}} \\
\leq C\Big(a, \gamma, \|f\|_{H^1}, \|g\|_{H^1}, \|u_0\|_{H^1}, \|v_0\|_{H^1}\Big) \\
 \left\| v(t) - e^{-\alpha t\p_x^3 - \gamma t} v_0  - \int_0^t  e^{(-\alpha\p_x^3 - \gamma )(t-r)}\rho_2(u,v)(r)\d r \right\|_{H^{1+a}} 
\\ \leq C\Big(a, \gamma, \|f\|_{H^1}, \|g\|_{H^1}, \|u_0\|_{H^1}, \|v_0\|_{H^1}\Big). 
\end{align*}
This bound extends to large times by breaking the time interval down into $\delta$-length pieces. Due to the dissipation, the norm over the short intervals decays over time so that the sum remains uniformly bounded. For details of the argument, see Section 6 in \cite{ET2}.
\end{proof} 
We now show that $U_t$ is asymptotically compact, i.e., for any bounded sequence $\big\{(u_{0,k}, v_{0,k})\big\}$ in $\dot{H}^1 \times H^1$ and sequence of times $t_k \to \infty$, the sequence $\big\{U(t_k)(u_{0,k},v_{0,k})\big\}_k$ has a convergent subsequence in $\dot{H}^1 \times H^1$. It suffices to consider sequences $\big\{(u_{0,k}, v_{0,k})\big\}$ which lie within the absorbing set $\mathcal{B}_0$. By Theorem \ref{MBforcedsmoothing}, for any $\alpha$ such that the resonant terms $\rho_1$ and $\rho_2$ are zero (i.e. $c_i, d_i \notin \Q$), we have 
\[ U_{t_k}(u_{0,k}, v_{0,k})  = (e^{-t_k \p^3_x - \gamma t_k} u_{0,k}, e^{-\alpha t_k \p^3_x - \gamma t_k} v_{0,k}) + N_{t_k}(u_{0,k}, v_{0,k}), \]
where $N_{t_k}(u_{0,k},v_{0,k})$ is in a ball in $H^{1+a} \times H^{1+a}$. Note we can take $a = \frac12 - $ for almost every $\alpha$. In the following, we assume $a = \frac12 -$. 

By Rellich's theorem, there is a subsequence of $\big\{ N_{t_k}(u_{0,k},v_{0,k})\big\}$ which converges in $H^1 \times H^1$. Furthermore
\begin{align*}
\| e^{-t_k \p^3_x - \gamma t_k} u_{0,n} \|_{H^1_x} + \|e^{-\alpha t_k \p^3_x - \gamma t_k} v_{0,n}\|_{H^1_x} \lesssim e^{-\gamma t_k}\Big(\|u_{0,n} \|_{H^1_x} + \|v_{0,n}\|_{H^1_x} \Big) \lesssim e^{-\gamma t_k}
\end{align*}
converges to zero uniformly as $k \to \infty$. Thus $U_{t_k}(u_{0,k}, v_{0,k})$ has a convergent subsequence and $U_t$ is asymptotically compact.  

To show that the attractor is compact in $H^{1+a} \times H^{1+a}$, it suffices, by Rellich's theorem, to show that it is is bounded in $H^{1 + a + \epsilon} \times H^{1 + a + \epsilon}$ for some $\epsilon >0$. To do this, choose $\epsilon>0$ small so that the nonlinear part of the solution lies in $H^{1 + a + \epsilon} \times H^{1 + a + \epsilon}$, e.g. take $\epsilon = (\frac12 - a)/2$. We show that the attractor is contained in a closed ball, say $\mathcal{B}_\epsilon$, in this space. 

Define $V_\tau = \overline{\cup_{t \geq \tau} U_t\mathcal{B}_0}$ so that the attractor is 
\[ \bigcap_{\tau \geq 0} \overline{ \bigcup_{t \geq \tau} U_t\mathcal{B}_0} = \bigcap_{\tau \geq 0} V_\tau. \]
Using the smoothing result again, elements in $V_\tau$ can be broken into two pieces -- the linear evolution which is converging uniformly to zero in $H^1$ by the argument above, and the nonlinear evolution which lives in some ball in $H^{1 + a + \epsilon} \times H^{1 + a + \epsilon}$. 

Thus as a subset of $\dot{H}^1 \times H^1$, the set $V_\tau$ is contained in a $\delta_\tau$-neighborhood $N_\tau$ of a ball $\mathcal{B}_{\epsilon}$ in $H^{1 + a + \epsilon} \times H^{1 + a + \epsilon}$. The uniform convergence of the linear parts to zero implies that $\delta_\tau \to 0$ as $\tau \to \infty$. Therefore the attractor is inside $\mathcal{B}_\epsilon$:
\[ \bigcap_{\tau \geq 0} V_\tau \subset \bigcap_{\tau \geq 0} N_\tau = \mathcal{B}_\epsilon. \]


\section{Trivial Attractor for $\gamma$, $\delta$ Large}

In this section, we show that when the damping is large relative to the forcing terms in the dissipative system \eqref{eq:MBforced}, the global attractor consists of a single function, namely the solution to the time-independent system. We focus on the $\alpha \neq 1$ case with a global attractor in $\dot{H}^1 \times H^1$, noting along the way where the argument differs for $\alpha = 1$ and the $\dot{L}^2 \times \dot{L}^2$ attractor. 

Consider the stationary version of the forced and weakly damped Majda-Biello system:
\begin{equation} \label{eq:MBstat}
\begin{cases}p_{xxx} + \gamma p + qq_x = f \\ \alpha q_{xxx} + \delta q  + (pq)_x = g.\end{cases}
\end{equation}
We will take $\gamma = \delta$ to simplify the notation; the arguments can be applied to the general case by replacing $\gamma$ by $\min\{\gamma,\delta\}$ in the estimates. The first step is to demonstrate the existence of a solution to \eqref{eq:MBstat} under certain conditions on $\gamma$, $f$, and $g$. 

\begin{prop}
If $\|f\|_{H^1} \ll \alpha^{1/3} \gamma^{4/3}$ and $\|g\|_{H^1} \ll \alpha^{1/2}\gamma^{4/3}$, then \eqref{eq:MBstat} has a unique solution on a ball in $H^2(\T)$. The same statement holds if $\|f\|_{L^2} \ll \alpha^{1/3}\gamma$ and $\|g\|_{L^2} \ll \alpha^{5/6} \gamma$. 
\end{prop}

\begin{proof}
The proof uses a fixed point argument. To construct the contraction operator, begin by taking the Fourier transform of the stationary system:
\begin{equation*}\begin{cases} -ik^3 p_k + \gamma p_k  + (qq_x)_k = f_k \\ -i \alpha k^3 q_k + \gamma q_k +  ((pq)_x)_k = g_k \end{cases}\end{equation*}

Define Fourier multiplier operators $\mathcal{M}_1$ and $\mathcal{M}_2$ as follows:
\begin{align*}
\mathcal{M}_1: w_k \mapsto \frac{w_k}{\gamma - ik^3} \qquad \qquad \mathcal{M}_2: w_k \mapsto \frac{w_k}{\gamma - i \alpha k ^3}.
\end{align*}
We have $\| \mathcal{M}_1 w \|_{H^{s+1}} \lesssim \frac{1}{\gamma^{2/3}}\|w\|_{H^{s}}$ and $\|\mathcal{M}_2w\|_{H^{s+1}} \lesssim \frac{1}{\alpha^{1/3}\gamma^{2/3}} \|w\|_{H^s}$. To see this, write
\begin{align*} 
\| \mathcal{M}_1 w\|_{H^{s+1}}  
= \left\| \frac{ \lb k \rb ^{s+1} w_k}{\gamma - i k^3} \right\|_{\ell^2} 
&\leq \left\| \frac{\lb k \rb}{\gamma - i k^3} \right\|_{\ell^\infty} \|w\|_{H^s} \\
&\leq \frac{1}{\gamma^{2/3}} \left\| \frac{\lb k \rb}{(\gamma - i k^3)^{1/3}} \right\|_{\ell^\infty} \|w\|_{H^s} 
\leq \frac{\sqrt{2}}{\gamma^{2/3}} \|w\|_{H^s}. 
\end{align*}
The constant in the last inequality is $\sqrt{2}$ and not $\max\{(1/\gamma)^{1/3}, \sqrt{2}\}$ because we're working with mean zero functions. The arguments go through without this assumption, but the power of $\gamma$ will change slightly. 
The other estimate is proved in the same way. Now notice that a solution to \eqref{eq:MBstat} must satisfy $ p = \mathcal{M}_1(f - qq_x)$. Substituting this into the evolution equation for $q$, we find that $q$ must satisfy
\begin{equation*}
q = \mathcal{M}_2( g - (pq)_x) = \mathcal{M}_2\Big( g - \big(\mathcal{M}_1(f - qq_x) q\big)_x\Big).
\end{equation*}
Let $T(q) = \mathcal{M}_2\Big( g - \big(\mathcal{M}_1(f - qq_x) q\big)_x\Big)$. We will find a fixed point of $T$.  
Estimate $T(q)$ as follows:
\begin{align*}
\| T(q) \|_{H^2} &\lesssim \frac{1}{\alpha^{1/3}\gamma ^{2/3}} \|g - (\mathcal{M}_1(f - qq_x)  q)_x \|_{H^1}
\leq \frac{1}{\alpha ^{1/3} \gamma ^{2/3}} \bigg( \|g\|_{H^1} + \|\mathcal{M}_1(f - qq_x)  q\|_{H^2} \bigg)\\
&\leq \frac{1}{\alpha^{1/3} \gamma ^{2/3}}  \bigg( \|g\|_{H^1} + \|\mathcal{M}_1(f - qq_x) \|_{H^2} \| q\|_{H^2} \bigg) \\
&\leq \frac{1}{\alpha^{1/3} \gamma^{2/3}} \bigg( \|g\|_{H^1} + \frac{1}{\gamma^{2/3}}\|(f - qq_x) \|_{H^1} \| q\|_{H^2}\bigg) \\
&\lesssim \frac{\|g\|_{H^1}}{\alpha^{1/3}\gamma^{2/3}} + \frac{\|q\|_{H^2}}{\alpha^{1/3} \gamma^{4/3}} \Big( \|f\|_{H^1}  + \|q\|_{H^2}^2 \Big). 
\end{align*}
Now we make the contraction estimate:
\begin{align*}
&\| T(w) - T(\tilde{w})\|_{H^2} 
= \left\| \mathcal{M}_2\bigg(\Big(\mathcal{M}_1(f- \tilde{w}\tilde{w}_x)  \tilde{w} - \mathcal{M}_1(f - ww_x) w\Big)_x \bigg)\right\|_{H^2}\\
&\lesssim \frac{1}{\alpha^{1/3}\gamma ^{2/3}} \Big\| \mathcal{M}_1(f- \tilde{w}\tilde{w}_x)  \tilde{w} - \mathcal{M}_1(f - ww_x) w \Big\|_{H^2} \\
&= \frac{1}{\alpha^{1/3}\gamma ^{2/3}} \Big\| \mathcal{M}_1(f- \tilde{w}\tilde{w}_x)  (\tilde{w}- w) + \mathcal{M}_1\big((w-\tilde{w})w_x + \tilde{w}(w-\tilde{w})_x\big)  w\Big\|_{H^2} \\
&\leq \frac{1}{\alpha^{1/3}\gamma ^{4/3}} \Big(  \|f - \tilde{w}\tilde{w}_x \|_{H^1} \| w - \tilde{w} \|_{H^2} + \| (w - \tilde{w})w_x\|_{H^1}\|w\|_{H^2} + \| \tilde{w} (w- \tilde{w})_x \|_{H^1} \|w\|_{H^2} \Big) \\
&\lesssim \frac{\|w-\tilde{w}\|_{H^2}}{\alpha^{1/3}\gamma ^{4/3}} \Big(  \|f \|_{H^1} + \|\tilde{w}\|_{H^2}^2 + \|w\|_{H^2}^2 + \| \tilde{w}\|_{H^2} \|w\|_{H^2} \Big). 
\end{align*}
Thus to close the contraction on a ball $\{ q \in H^2 : \|q\|_{H^2} \leq R\}$, two inequalities must hold: 
\[ \frac{\|g\|_{H^1}}{\alpha^{1/3} \gamma^{2/3}} + \frac{\|f\|_{H^1} R + R^3}{\alpha^{1/3}\gamma^{4/3}} \leq \frac{R}{C} \qquad \text{ and } \qquad
\frac{\|f\|_{H^1} + 3R^2}{\alpha^{1/3} \gamma^{4/3}} < \frac1C. \]
These can be satisfied by taking $R  = \frac1{2 \sqrt{C}}\alpha ^{1/6}\gamma^{2/3}$ as long as $\|f\|_{H^1} < \frac1{4C} \alpha^{1/3}\gamma^{4/3}$ and $\|g\|_{H^1} \leq \frac{1}{4{C}^{3/2}} \alpha^{1/2} \gamma^{4/3}$.

The proof for the $L^2$ statement is similar. The only difference is that one uses the estimates $\|\mathcal{M}_1 f\|_{H^{s+2}} \leq \frac{1}{\gamma^{1/3}}\|f\|_{H^s}$ and $\|\mathcal{M}_2g\|_{H^{s+2}} \leq \frac{1}{\alpha^{2/3} \gamma^{1/3}} \|g\|_{H^s}$. 
\end{proof}

\begin{remark}
If $g=0$, the existence of a stationary solution is trivial; the solution is $(\mathcal{M}_1(f), 0)$. The convergence arguments are also greatly simplified in this case. 
\end{remark}

We now show that solutions to \eqref{eq:MBforced} converge to the stationary solution under certain conditions on $f$, $g$, and $\gamma$, implying that the attractor is trivial. Let $(u,v)$ be a solution of the dissipative Majda-Biello system \eqref{eq:MBforced} and define $y = u - p$ and $z = v - q$. We show that if $f$ and $g$ are small relative to $\gamma$, then $y$ and $z$ converge to zero in $H^1$ if $u,v \in H^1$. Notice that $y$ and $z$ satisfy 
\begin{align} \label{eq:yzeq}
\begin{cases}
y_t + y_{xxx} + \gamma y + zz_x + (qz)_x = 0 \\
z_t + \alpha z_{xxx} + \gamma z + (yz)_x + (pz+ qy)_x = 0.
\end{cases}
\end{align}
Recall that $\int u^2 + v^2 \; \d x \;$ and $\int u_x^2 + \alpha v_x^2 - uv^2 \; \d x\;$ are conserved for the original Majda-Biello system.
Our estimates will be based on these conservation laws. Recall $E_3 = \int y^2 + z^2 \; \d x$. Then we have 
\begin{align*}
\frac{\p}{\p t} E_3  &= -2 \int y\big(\gamma y + (qz)_x\big)  + z\big(\gamma z + (pz + qy)_x\big) \; \d x \\
&= -2\gamma E_3 - 2 \int qyz_x + q_x yz +  p_xz^2 + pzz_x + qy_xz + q_xyz \; \d x \\
&= -2 \gamma E_3 - 2 \int 2 q_x yz + q(yz)_x + \frac12 p_xz^2 \; \d x \\
&= -2 \gamma E_3 -  \int 2q_xyz + p_x z^2 \; \d x \\
&\leq -2 \gamma E_3 + \|p_x\|_{L^\infty} \|z\|_{L^2}^2 + 2 \|q_x\|_{L^\infty} \|y\|_{L^2} \|z\|_{L^2} \\
&\leq (-2 \gamma + C\|p\|_{H^2} + C\|q\|_{H^2}) E_3. 
\end{align*}
So to ensure that $E_3 \to 0$ as $t \to \infty$, i.e. that $(u,v) \to (p,q)$ in $L^2$, we need $C\|p\|_{H^2} + C\|q\|_{H^2} < 2 \gamma$. The contraction argument for the existence of $q$ was carried out in a ball of radius $R = \frac{1}{2\sqrt{C}} \alpha^{1/6} \gamma^{2/3}$, so we have $C\|q\|_{H^2} < \gamma$ as long as $\gamma^{1/3} > \frac{\sqrt{C}\alpha^{1/6}}{2}$.  Also notice that 
\[C\|p\|_{H^2} \leq \frac{C}{\gamma^{2/3}} \left( \|f\|_{H^1} + \|qq_x\|_{H^1} \right) \leq \frac{C}{\gamma ^{2/3}} (\|f\|_{H^1} +\|q\|_{H^2}^2).\] This is bounded by $\gamma$ when $C\|f\|_{H^1} < \frac{\gamma^{5/3}}{2}$ and $8 \gamma > \alpha$. So we have a stationary solution and $L^2$ convergence to it whenever $\| f\|_{H^1}, \|g\|_{H^1} \ll \gamma ^{4/3}$ and $\gamma > \sqrt{C^3\alpha}/8$. The same holds when $\| f\|_{L^2}, \|g\|_{L^2} \ll \gamma$ and $\gamma > \sqrt{C^3\alpha}/8$, which completes the proof for the $\alpha = 1$ case. 

For the $H^1$ convergence, we use a modification of the Hamiltonian integral $E_4$:
\[ H_4 = \int y_x^2 + \alpha z_x^2 - yz^2 - 2qyz - pz^2 \; \d x. \]
The last two terms are added to make the time derivative well-behaved. Calculating this derivative, we find
\begin{align*}
\frac{\p }{\p t} &H_4 
= -2\int y_x\Big(\gamma y_x + (qz)_{xx}\Big)\; \d x 
- 2 \alpha \int z_x \Big( \gamma z_x  + (pz + qy)_{xx}  \Big)\; \d x \\
&+  \int  z^2\Big(\gamma y + (qz)_x\Big) \; \d x 
+ 2\int  yz \Big( \gamma z + (pz + qy)_x\Big) \; \d x \\
&+ 2\int qy\Big( \alpha z_{xxx} + \gamma z+(yz)_x + (pz+qy)_x\Big) \; \d x \\ 
&+ 2\int qz \Big( y_{xxx} + \gamma y + zz_x + (qz)_x \Big) \; \d x  
+ 2\int pz \Big(  \alpha z_{xxx} + \gamma z + (yz)_x + (pz + qy)_x \Big) \; \d x \\ 
&\quad = -2 \gamma H_4 + \gamma \int  yz^2 \; \d x .
\end{align*}
Notice that 
\[ \int yz^2 \; \d x \lesssim \|y\|_{H^1} \|z\|_{L^2}^2 \lesssim e^{-at}\]
by the embedding $L^\infty \hookrightarrow H^1$, the bound on the $H^1$ norm of $y$ (which follows from Lemma \ref{H1bound}), and the decay of the $L^2$ norm of $z$. Here $a = -2\gamma + C\|p\|_{H^2} + C\|q\|_{H^2} > 0$. Thus we have 
\[ \p_t[ e^{2\gamma t} H_4] \lesssim e^{-at}. \]
Integrating this inequality gives $H_4(t) \lesssim e^{-2 \gamma t}$.  
Furthermore, since $\|y\|_{L^2}^2 + \|z\|_{L^2}^2 \to 0$ as $t \to \infty$ and the $L^2$ norms of $p$, $q$, $y$, and $z$ are bounded, we have  
\[ \left| \int yz^2 + 2qyz + pz^2 \; \d x \right| \lesssim \|y\|_{H^1}\|z\|_{L^2}^2 + 2 \|q\|_{H^1} \|y\|_{L^2}\|z\|_{L^2} + \|p\|_{H^1}\|z\|_{L^2}^2 \to 0 \text{ as } t \to \infty. \]
Thus we have 
\[ \left| \int y_x^2 + \alpha z_x^2 \; \d x \right| \leq |H_4| + \left| \int yz^2 + 2qyz + pz^2 \; \d x \right| \to 0 \text{ as } t \to \infty. \] 
This, along with the $L^2$ convergence show above, implies that $ y = u - p$ and $z = v - q$ converge to zero in $H^1$. 
\section{Proofs of Smoothing Estimates} \label{estimateproofs}

To begin, we state a standard calculus lemma which will be used repeatedly. See, e.g., the appendix of \cite{ET2} for proofs of similar results. 

\begin{lem} \label{sumest}
\begin{enumerate}[(a)]
\item If $\beta \geq \gamma > 1$, then
\[ \sum_n \frac{1}{\lb n - k_1 \rb ^\beta \lb n - k_2 \rb^\gamma} \lesssim \lb k_1 - k_2 \rb ^{-\gamma}. \]
\item If $\beta > \frac13$, then 
\[ \sum_n \frac{1}{ \lb n^3 + an^2 + bn + c \rb^\beta} \lesssim 1, \]
with implicit the constant independent of $a$, $b$, and $c$. 
\end{enumerate}
\end{lem}
The proofs for the cases where $\alpha = \frac{q^2}{3p(p-q) + q^2}$ are much easier than those for the general cases, and are therefore not explicitly included. 

\subsection{Proof of Proposition \ref{B1}} \label{B1proof}

By symmetry, it suffices to consider $|k_1| \gtrsim |k_2|$. Then we need to bound 
\begin{align*}
 \left\| \sum_{\substack{k_1+k_2 = k \\ |k_1| \gtrsim |k_2|}}^* \frac{\langle k \rangle^{1+s_1} u_{k_1}v_{k_2}}{k^3 - \alpha k_1^3 - \alpha k_2^3} \right\|_{\ell^2_k}
\end{align*}

\noindent\textbf{Case 1.} $|k_1 - c_1k| \geq \frac12$ and $|k_1 - c_2k| \geq \frac12$ 

Note that 
\begin{align*}
|k^3 - \alpha k_1^3 -\alpha k_2^3| = |3\alpha k(k_1 - c_1k)(k_1 - c_2k)| 
\gtrsim |k| \cdot \max\{|k_1 - c_1k|, |k_1 - c_2 k|\} \\
\gtrsim |k|\cdot (c_1 - c_2)|k| 
\gtrsim |k|^2. 
\end{align*}
 Then using $|k_1| \gtrsim k$, the assumption that $s_1-s -1 \leq 0$, and Young's inequality, we find
\begin{align*}
\| B_1(u,v) \|_{H^{s_1}_x} &\lesssim \left\| \sum_{\substack{k_1+k_2 = k \\ |k_1| \gtrsim |k_2|}}\langle k \rangle ^{s_1-s-1} \frac{|u_{k_1}| \langle k_1\rangle^s |v_{k_2}|\langle k_2\rangle^s}{\langle k_2\rangle^s} \right\|_{\ell^2_k} \\
&\leq \left\| \sum_{k_1+k_2 = k} \frac{|u_{k_1}| \langle k_1\rangle^s |v_{k_2}|\langle k_2\rangle^s}{\langle k_2\rangle^s} \right\|_{\ell^2_k} 
\leq \| u\|_{H^s_x} \left\| \frac{v_{k} \langle k \rangle ^s}{\langle k \rangle ^s} \right \|_{\ell^1_k}\\
& \leq \| u\|_{H^s_x} \| v\|_{H^s_x} \| \langle k \rangle ^{-s}\|_{\ell^2_k} 
\lesssim \| u\|_{H^s_x} \| v\|_{H^s_x}.
\end{align*} 

\noindent\textbf{Case 2.} $|k_1 - c_1k| < \frac12$ or $|k_1 - c_2k| < \frac12$ 

Assume that $|k_1 - c_1k| < \frac12$. The other case is parallel. Note that $| k_1 - c_1k| \geq M_{\epsilon_0} |k|^{-1 - \nu_{c_1} - \epsilon_0}$ for any $\epsilon_0 > 0$, where $\nu_{c_1}$ is the minimal type index of $c_1$. This holds because
\[ | k_1 - c_1 k | = |k|\left| \frac{k_1}{k} - c_1 \right| \geq |k|\frac{M_{\epsilon_0} }{|k|^{2 + \nu_{c_1} + \epsilon_0}}\]
for any positive $\epsilon_0$ by definition of the minimal type.  
Therefore
\begin{align*}
|k^3 -\alpha k_1^3 - \alpha k_2^3 | = 3\alpha |k(k_1-c_1k)(k_1-c_2k)| 
&\geq 3\alpha M_{\epsilon_0} |k|^{-\nu_{c_1}-\epsilon_0}\Big((c_1-c_2)|k| - \frac12\Big) \\
&\gtrsim |k|^{1-\nu_{c_1} - \epsilon_0}.
\end{align*}
In this region there is only one term in the sum -- the one with $k_1 \simeq c_1k$ and $k_2 \simeq (1-c_1)k$. Using Cauchy-Schwartz with the fact that $|k| \approx |k_1| \approx |k_2|$, we get for this part of the sum
\begin{align*}
\|B_1(u,v)\|_{H^{s_1}_x} \lesssim \left\| \langle k \rangle ^{s_1 + \nu_{c_1} + \epsilon_0} u_{k_1} v_{k_2} \right\|_{\ell^2_k} 
\lesssim \left\| \langle k \rangle ^{(s_1 + \nu_{c_1} + \epsilon_0)/2} u_{k}\right\|_{\ell^4_k} \left\| \langle k \rangle ^{(s_1 + \nu_{c_1} + \epsilon_0)/2} v_{k}\right\|_{\ell^4_k} \\
\lesssim \left\| \langle k \rangle ^{(s_1 + \nu_{c_1} + \epsilon_0)/2} u_{k}\right\|_{\ell^2_k} \left\| \langle k \rangle ^{(s_1 + \nu_{c_1} + \epsilon_0)/2} v_{k}\right\|_{\ell^2_k} 
\lesssim \| u\|_{H^s_x}\| v\|_{H^s_x},
\end{align*}
where the third inequality holds when $s_1 -s < s - \nu_{c_1}$. 

\subsection{Proof of Proposition \ref{B2}}

Write 
\begin{align*}
\| B_2(u,v) \|_{H^{s_1}_x} &\lesssim \left\| \sum_{k_1 + k_2 = k}^* \frac{\lb k \rb^{1 + s_1} u_{k_1}v_{k_2}}{\alpha k^3 - k_1^3 - \alpha k_2^3} \right\|_{\ell^2_k}.
\end{align*}

\noindent\textbf{Case 1.} $|k_1 - d_1k|, |k_1 - d_2k| \geq \epsilon$ and $|k_1| \geq \epsilon|k|$ 

In this case, $|k_1| \gtrsim |k_2|$ and  
\[ |\alpha k^3 - k_1^3 -\alpha k_2^3| = |(1-\alpha)k_1(k_1-d_1k)(k_1-d_2k)| \gtrsim |kk_1|. \]
The argument in Case 1 of the $B_1$ estimate gives the bound when $s_1 - s \leq 1$ and $s > \frac12$. 

\noindent\textbf{Case 2.} $|k_1| \leq \epsilon|k|$ 

Recall that $k_1 \neq 0$ since $u$ is mean zero and write
\[ k_1 = \mu k  \text{  for some  } |\mu| \in [1/|k|, \epsilon]. \]
Then 
\begin{align*}
|\alpha k^3 - k_1^3 - \alpha (k-k_1)^3 | &= |\mu k^3||3 \alpha (1-\mu) - \mu^2(1-\alpha)| 
\geq k^2 \left(3\alpha(1-\epsilon) - \epsilon^2(1-\alpha) \right)  \gtrsim k^2.
\end{align*}
Apply the argument from Case 1 of the $B_1$ proof again to get the bound when $s_1 -s \leq 1$ and $s > \frac12$. 

\noindent\textbf{Case 3.} $|k_1 - d_1k| \leq \epsilon$ or $|k_1 - d_2k| \leq \epsilon$, with $|k_1| \geq \epsilon |k|$ 

Assume $|k_1 - d_1k| \leq \epsilon$. The other case is parallel. Note that in this region $|k| \sim |k_1| \sim |k_2|$ and the values of $k_1$ and $k_2$ are determined by $k$. We need only bound the following sum, where $k_1 \simeq d_1k$ and $k_1 + k_2 = k$,
\begin{align*}
\left( \sum_{\substack{k }} \lb k \rb ^{2(s_1 + \nu_{d_1} + \epsilon_0)} u_{k_1}^2v_{k_2}^2 \right)^{1/2}
&\lesssim \left\| \lb k \rb ^{s_1 + \nu_{d_1} + \epsilon_0} u_{k}^2 \right\|_{\ell^2_k}^{1/2}\left\| \lb k \rb ^{s_1 + \nu_{d_1} + \epsilon_0} v_{k}^2 \right\|_{\ell^2_k}^{1/2} \\
&= \left\| \lb k \rb ^{(s_1 + \nu_{d_1} + \epsilon_0)/2} u_{k} \right\|_{\ell^4_k}\left\| \lb k \rb ^{(s_1 + \nu_{d_1} + \epsilon_0)/2} v_{k} \right\|_{\ell^4_k} \\
&\leq\left\| \lb k \rb ^{(s_1 + \nu_{d_1} + \epsilon_0)/2} u_{k} \right\|_{\ell^2_k}\left\| \lb k \rb ^{(s_1 + \nu_{d_1} + \epsilon_0)/2} v_{k} \right\|_{\ell^2_k} \\
&\lesssim \|u\|_{H^s_x} \|v\|_{H^s_x} . 
\end{align*}
The last inequality holds when $s_1 + \nu_{d} + \epsilon_0 \leq 2s$, i.e. when $s_1 - s < s -\nu_d$.  

\subsection{Proof of Proposition \ref{R1}}

We need to establish
\begin{align} \label{R1statement} 
\left\| \sum_{k_1 + k_2 + k_3 = k}^* \frac{(k_1 + k_2)u_{k_1}v_{k_2}w_{k_3}}{(k_1 + k_2 - c_1k)(k_1+k_2 - c_2 k )} \right\|_{X^{s_1,b-1}_1} \lesssim \|u\|_{X^{s,1/2}_1} \|v\|_{X^{s,1/2}_\alpha}\|w\|_{X^{s,1/2}_\alpha}. 
\end{align}
Define the following functions
\begin{align*} 
f(k, \tau) &= \lb k \rb^s \lb \tau -  k^3 \rb^{1/2} u_{k} \qquad \\
g(k, \tau) &= \lb k \rb^s \lb \tau -  \alpha k^3 \rb^{1/2} v_{k} \qquad \\
h(k, \tau) &= \lb k \rb^s \lb \tau -  \alpha k^3 \rb^{1/2} w_{k}. 
\end{align*}
Then (\ref{R1statement}) amounts to showing that 
\begin{align} \label{R1statement2}
\left\| \int\limits_{\sum \tau_i = \tau} \sum^*\limits_{\sum k_i = k} M \; f(k_1,\tau_1) g(k_2, \tau_2) h(k_3, \tau_3)  \; \d\tau_1 \d \tau_2 \d \tau_3 \right\|_{L^2_\tau\ell^2_k}^2 
\lesssim \|f\|_{L^2_\tau \ell^2_k}^2 \|g\|_{L^2_\tau\ell^2_k}^2 \|h\|_{L^2_\tau\ell^2_k}^2 
\end{align}
where the multiplier $M= M(k_1, k_2, k_3, k, \tau_1, \tau_2, \tau_3, \tau)$ is
\begin{align*}
M   = \frac{(k_1+ k_2)\lb k \rb^{s_1}\lb k_1 \rb^{-s} \lb k_2 \rb ^{-s} \lb k_3 \rb^{-s}}{(k_1 + k_2 - c_1k)(k_1+k_2 - c_2 k )\lb \tau - k^3 \rb^{1-b}\lb \tau_1 - k_1^3 \rb^{1/2} \lb \tau_2 - \alpha k_2^3 \rb^{1/2} \lb \tau_3 - \alpha k_3^3 \rb ^{1/2}}. 
\end{align*}

Apply Cauchy-Schwartz in the $\tau_1$, $\tau_2$, $\tau_3$, $k_1$, $k_2$, and $k_3$ variables to bound the left-hand side of (\ref{R1statement2}) by 
\begin{align*}
\sup_{k, \tau} \bigg( \int\limits_{\sum \tau_i = \tau} \sum^*\limits_{\sum k_i = k} M^2 \bigg)  
\left\| \int\limits_{\sum \tau_i = \tau} \sum^*\limits_{\sum k_i = k} f^2(k_1,\tau_1) g^2(k_2, \tau_2) h^2(k_3, \tau_3)  \; \d\tau_1 \d \tau_2 \d \tau_3 \right\|_{L^1_\tau\ell^1_k} 
\end{align*}
Using Young's inequality twice bounds the $L^1\ell^1$ norm above by $\|f\|_{L^2_\tau \ell^2_k}^2 \|g\|_{L^2_\tau\ell^2_k}^2 \|h\|_{L^2_\tau\ell^2_k}^2$. Thus it suffices to show that the supremum on the left is finite. We can further simplify matters by repeatedly using the calculus estimate 
\[ \int_\mathbb{R} \frac{1}{\lb x \rb ^{\beta} \lb x -b \rb} \d x \lesssim \lb b \rb ^{-\beta},\] which holds for $\beta \in (0,1]$ (see \cite{ET2} for a proof), to eliminate the $\tau$ dependence and bound the supremum by 
\begin{align*}
\; \sup_k \;  \sum_{k_1,k_2 }^* \frac{\lb k \rb ^{2s_1} \lb k_1 \rb ^{-2s} \lb k_2 \rb ^{-2s} \lb k - k_1 -k_2 \rb^{-2s} |k_1 + k_2|^2}{(k_1 + k_2 - c_1k)^2(k_1+k_2 - c_2k)^2 \lb k^3 -k_1^3 - \alpha k_2^3 - \alpha (k- k_1 - k_2)^3\rb ^{2 -2b-}},
\end{align*}
or equivalently, using the change of variables $k_2 \mapsto n - k_1$,
\begin{align*}
\; \sup_k \;  \sum_{k_1,n }^* \frac{\lb k \rb ^{2s_1} \lb k_1 \rb ^{-2s} \lb n-k_1 \rb ^{-2s} \lb n-k\rb^{-2s} n^2}{(n - c_1k)^2(n- c_2k)^2 \lb k^3 -k_1^3 - \alpha (n-k_1)^3 - \alpha (k - n)^3\rb ^{2 -2b-}}.
\end{align*}
We will show that this supremum is finite by considering a number of cases. In the following, to simplfiy notation we will write $2-2b$ instead of the technically correct $2 - 2b -$. Since we take $b = \frac12 + $, this $\epsilon$-difference has no effect on the calculations. 

\noindent\textbf{Case 1.} $k_1 = k$ 

In this case, the supremum becomes 
\[ \; \sup_k \;  \sum_{n}^* \frac{\lb k \rb ^{2s_1 - 2s} \lb n-k\rb ^{-4s} n^2}{(n - c_1k)^2(n- c_2k)^2}. \]

\noindent\textbf{Case 1.1.} $kn > 0$ 

Since $c_2 < 0$, we cancel $n^2$ with $(n - c_2k)^2$ to obtain
\[ \; \sup_k \;  \sum_{n}^* \frac{\lb k \rb ^{2s_1 - 2s} \lb n-k\rb ^{-4s}}{(n - c_1k)^2}. \]

If $| n - c_1k| \geq \epsilon$, with $\epsilon$ small but fixed, then the supremum is bounded by 
\begin{align*}
\; \sup_k \;  \lb k \rb ^{2s_1 - 2s} \sum_{n} \frac{\lb n-k\rb ^{-4s}}{\lb n - c_1k\rb^2} 
\lesssim  \; \sup_k \;  \lb k \rb ^{2s_1 - 2s} \lb (c_1 - 1)k \rb^{-2} 
\lesssim \; \sup_k \;  \lb k \rb ^{2s_1 - 2s - 2},
\end{align*}
which is finite for $s_1 -s \leq 1$. In the first inequality, we used Lemma \ref{sumest}(a). 

If $|n - c_1k| < \epsilon$, then there's only one term in the sum since $n \simeq c_1$, and we have $|n-k| \gtrsim |k|$. Using the minimal type index, the supremum is bounded by 
\begin{align*}
\; \sup_k \;  \lb k \rb ^{2s_1 - 6s + 2 +2\nu_{c_1} + 2\epsilon_0},
\end{align*}
which is finite when $s_1 - s  < 2s - 1 - \nu_{c_1}$. 

\noindent\textbf{Case 1.2.} $kn < 0$ 

For this case, cancel $n^2$ with $(n - c_1k)^2$ and repeat the argument from Case 1.1. 

\noindent\textbf{Case 1.3.} $kn = 0$ 

The supremum is immediately bounded in this case. 

\noindent\textbf{Case 2.} $kn > 0$ with $k_1 \neq k$ 

In this region, the supremum is bounded by 
\begin{align*}
\; \sup_k \;  \sum_{k_1,n }^* \frac{\lb k \rb ^{2s_1} \lb k_1 \rb ^{-2s} \lb n-k_1 \rb ^{-2s} \lb n-k\rb^{-2s}}{(n - c_1k)^2\lb k^3 -k_1^3 - \alpha (n-k_1)^3 - \alpha (k - n)^3\rb ^{2 -2b}}.
\end{align*}

\noindent\textbf{Case 2.1} $|n - c_1k| \geq \epsilon |k|$ 

Here the supremum is bounded by 
\begin{align*}
\sup_k \;  \lb k \rb ^{2s_1 -2} \sum_{k_1,n } \lb k_1 \rb ^{-2s} \lb n-k_1 \rb ^{-2s} \lb n-k\rb^{-2s} 
\lesssim \; \sup_k \; \lb k \rb ^{2s_1 -2 - 2s} 
<  \infty 
\end{align*}
for $s_1 - s \leq 1$. This estimate comes from applying Lemma \ref{sumest}(a) repeatedly. 

\noindent\textbf{Case 2.2.} $\epsilon \leq |n-c_1k| < \epsilon |k|$ 

Note that $|n| \in \big((c_1 - \epsilon)|k|, (c_1 + \epsilon)|k|\big)$. Choose $\epsilon < c_1 - 1$ so that $|n - k| \gtrsim |k|$. The supremum is then bounded  by 
\begin{align*}
\; \sup_k \;  \lb k \rb ^{2s_1 -2s} \sum_{k_1,|n| \geq |k| }  \lb k_1 \rb ^{-2s} \lb n-k_1 \rb ^{-2s} 
\lesssim & \; \sup_k \;  \lb k \rb ^{2s_1 -2s} \sum_{|n| \geq |k| } \lb n \rb ^{-2s} 
\lesssim &  \; \sup_k \;  \lb k \rb ^{2s_1 -4s + 1}. 
\end{align*}
This is finite when $s_1 - s \leq s - \frac12$. 

\noindent\textbf{Case 2.3.} $|n - c_1k| < \epsilon$ 

\noindent\textbf{Case 2.3a.} $|k_1|, |k_1-n| \geq \epsilon |k|$ 

In this case, the supremum is bounded by 
\begin{align*}
&\; \sup_k \;  \lb k \rb ^{2s_1 + 2  + 2 \nu_{c_1} + 2\epsilon_0 -6s} \sum_{\substack{k_1 \\ (n \simeq c_1k)}} \lb k^3 -k_1^3 - \alpha(n-k_1)^3 + \alpha (n-k)^3 \rb^{-(2-2b)} \\
\lesssim & \; \sup_k \;  \lb k \rb ^{2s_1 + 2  + 2 \nu_{c_1} + 2\epsilon_0 -6s},
\end{align*}
which is finite for $s_1 - s < 2s - 1 -\nu_{c_1}$. This estimate comes from Lemma \ref{sumest}(b).  

\noindent\textbf{Case 2.3b.} $|k_1| < \epsilon |k|$ 

Note that in this case 
\[|k_1 - n| \geq c_1|k| - |n-c_1k| - |k_1| > (c_1 - \epsilon)|k| - \epsilon\]
so that $|k_1 - n| \gtrsim |k|$. Recall $k_1 \neq 0$ by the mean zero assumption on $u$, and write
\[ n = c_1 k + \delta \text{  for some } |\delta| < \epsilon, \qquad \qquad k_1 = \mu k \text{  for some  } |\mu| \in [1/|k|, \epsilon). \]
Then use the fact that $1- \alpha  = 3\alpha c_1 (c_1 - 1)$ to calculate that
\begin{align*}
& \quad |k^3 - k_1^3 -\alpha(n-k_1)^3 + \alpha (n-k)^3 | \\
& = |k-k_1|\left|\mu k^2[(1-\alpha)(1+ \mu) + 3 \alpha c_1] + 3\alpha \delta[1 + \mu - 2 c_1]k - 3 \alpha \delta^2\right| \\
& \geq |k-k_1| \left[ (3\alpha c_1 + (1 - \alpha)(1 - \epsilon))|k|  - 3 \alpha \epsilon(2c_1 + \epsilon - 1)|k| - 3 \alpha \epsilon ^2 \right] 
\gtrsim |k_1 - k| |k|.
\end{align*}
Using Lemma \ref{sumest}(a) again, the supremum is bounded by
\begin{align*}
\; &\sup_k \;  \lb k \rb ^{2s_1 + 2 + 2 \nu_{c_1} +2 \epsilon_0 -4s} \sum_{ k_1} \frac{ \lb k_1 \rb ^{-2s}}{\lb (k-k_1)k \rb ^{2 - 2b}} \\
\lesssim &\sup_k \;  \lb k \rb ^{2s_1 + 2 + 2 \nu_{c_1} +2 \epsilon_0 -4s - (2-2b)} \sum_{ k_1} \frac{ \lb k_1 \rb ^{-2s}}{\lb k-k_1 \rb ^{2 - 2b}} 
\lesssim  \; \sup_k \;  \lb k \rb ^{2s_1 + 2 + 2 \nu_{c_1} +2 \epsilon_0 -4s - 2(2-2b)}, 
\end{align*}
which is finite when $s_1 - s < s + 1 - 2b - \nu_{c_1}$. 

\noindent\textbf{Case 2.3c.} $|n - k_1| < \epsilon |k|$ 

In this case we have $|k_1| \geq |n| -|n - k_1| \geq (c_1 - \epsilon)|k| - \epsilon$
so that  $|k_1| \geq |k|$. 
The supremum is bounded by 
\begin{align*}
\; \sup_k \;  \lb k \rb ^{2s_1 +2 + 2\nu_{c_1} + 2\epsilon_0 -4s}\sum_{\substack{k_1 \\ n \simeq c_1k }} \frac{ \lb n-k_1 \rb ^{-2s}}{\lb k^3 -k_1^3 - \alpha (n-k_1)^3 - \alpha (k - n)^3\rb ^{2 -2b}}.
\end{align*}

We may assume, since $(k,k_1,n) \to (-k,-k_1,-n)$ is a symmetry for the supremum, that $k_1 >0$. Then in our case of $kn >0$, we must have $k,n >0$, since otherwise $|k_1 - n| > |n| \simeq c_1|k|$. 

Notice that the following three inequalities hold:
\begin{align*}
k^2 + k_1k+k_1^2 \geq k^2, \qquad \qquad
3\alpha n  > 0, \qquad \text{and} \qquad
k_1 - (n-k) \geq (1-\epsilon)k - 2\epsilon.  
\end{align*}
Thus we have
\[ (1-\alpha)(k^2 + kk_1 + k_1^2) - 3\alpha n (n-n-k_1) \gtrsim k^2, \]
which implies that
\[ k^3 -k_1^3 - \alpha (n-k_1)^3 - \alpha (k-n)^3 \gtrsim |k-k_1|k^2.\]
The supremum is therefore bounded by
\begin{align*}
\; \sup_k \;  \lb k \rb ^{2s_1 +2 + 2\nu_{c_1} + 2\epsilon_0 -4s}\sum_{\substack{k_1 \\ n \simeq c_1k }}\frac{ \lb n-k_1 \rb ^{-2s}}{\lb k^2(k-k_1) \rb ^{2 -2b}}
\lesssim \; \sup_k \;  \lb k \rb ^{2s_1 +2 + 2\nu_{c_1} + 2\epsilon_0 -4s - 6 + 6b},
\end{align*}
which is finite if $s_1 - s < s+ 2 - 3b - \nu_{c_1}$. 

\noindent\textbf{Case 3.} $kn < 0$ and $k_1 \neq k$ 

In this case, the supremum can be bounded by 
\[ \; \sup_k \;  \lb k \rb^{2s_1 - 2s} \sum_{\substack{k_1 > 0 \\ n}}^* \frac{ \lb k_1 \rb^{-2s} \lb n - k_1\rb ^{-2s}}{(n - c_2k)^2\lb k^3 -k_1^3 \alpha(n-k_1)^3 - \alpha(k-n)^3 \rb^{2-2b}}. \]

\noindent\textbf{Case 3.1.} $|n - c_2k | \geq \epsilon |k|$ 

If $s > \frac12$ and $s_1 -s \leq 1$, the supremum is bounded by 
\begin{align*}
\; \sup_k \;  \lb k \rb^{2s_1 - 2s  -2 } \sum_{\substack{k_1 > 0 \\ n}} \lb k_1 \rb^{-2s} \lb n - k_1\rb ^{-2s}
\lesssim & \; \sup_k \;  \lb k \rb^{2s_1 - 2s  -2 } <  \infty, 
\end{align*}
. 

\noindent\textbf{Case 3.2.} $\epsilon \leq |n - c_2 k | < \epsilon |k|$ 

Note that $|n| \geq (|c_2| - \epsilon) |k| \approx |k|$. When $s_1 -s \leq s - \frac12$, the supremum is bounded by 
\begin{align*}
\sup_k \;  \lb k \rb^{2s_1 - 2s } \sum_{\substack{k_1 > 0 \\ |n| \gtrsim |k|}} \lb k_1 \rb^{-2s} \lb n - k_1\rb ^{-2s} 
< \; \sup_k \;  \lb k \rb^{2s_1 - 2s } \sum_{ |n| \gtrsim |k|} \lb n \rb ^{-2s} 
<  \infty.
\end{align*}

\noindent\textbf{Case 3.3.} $|n - c_2k| < \epsilon $ 

\noindent\textbf{Case 3.3a.} $|k_1|, |k_1 - n| \geq \epsilon |k|$ 

As in Case 2.3a, the supremum is finite if $s_1 -s < 2s - 1 - \nu_{c_1}$. 

\noindent\textbf{Case 3.3b.} $|k_1| < \epsilon |k|$ 

Here $|k_1 - n| \geq (|c_2| - \epsilon)|k| - \epsilon$, so $|k_1 - n | \gtrsim |k|$. 
Write 
\[ k_1 = \mu k \text{  for some  } |\mu| \in [1/|k|, \epsilon ), \qquad  n = c_2 k + \delta \text{  for some  } |\delta| < \epsilon. \]
Expanding the resonance equation with this notation gives
\begin{align*}
& \quad |k^3 - k_1^3 - \alpha (n - k_1)^3 - \alpha (k-n)^3| \\
&=  |k - k_1| \left|  \mu k^2(3\alpha c_2^2 + \mu(1-\alpha)) + 3 \alpha \delta(1 + \mu - 2c_2)k - 3 \alpha \delta^2 \right| \\
&\geq |k - k_1| \left[ \left( 3 \alpha c_2^2 - \epsilon(1-\alpha) - 3\alpha\epsilon[1 + \epsilon - 2 c_2] \right) |k| - 3 \alpha \epsilon^2 \right] 
 \gtrsim |k-k_1| |k|.
\end{align*}
Notice that, depending on $\alpha$, we may have $|c_2 | \ll 1$, but by choosing $\epsilon$ small enough, we can ensure that 
\[  \left[ \left( 3 \alpha c_2^2 - \epsilon(1-\alpha) - 3\alpha\epsilon[1 + \epsilon - 2 c_2] \right) |k| - 3 \alpha \epsilon^2 \right] \gtrsim |k|\] 
to get the last inequality above. 
Then as in Case 2.3b, the supremum is finite if $s_1 - s < s + 1 - 2b - \nu_{c_1}$. 

\noindent\textbf{Case 3.3c.} $|k_1 -n |\leq \epsilon |k|$ 

Note $|k_1| \geq |n| - |k_1 - n| \geq (|c_2| - \epsilon)|k| - \epsilon$, so for $\epsilon$ small enough, $|k_1| \geq \epsilon |k|$. 
Write 
\[n  - k_1 = \mu k \text{  for some  } |\mu| \in \{0\} \cup [1/|k|, \epsilon], \qquad  n = c_2k + \delta \text{  for some  } |\delta| < \epsilon. \]
With this notation,
\begin{align*}
|k^3 - \alpha(n-k_1)^3 - \alpha (k - n)^3| 
 = \left|\alpha(c_2^3 -  \mu^3 )k^3 + 3\alpha \delta (1-c_2)^2 k^2 - 3\alpha \delta^2 (1-c_2) k  + \alpha\delta^3  \right| \\ 
\geq \alpha(|c_2|^3 - \epsilon^3)|k|^3 - 3\alpha \epsilon(1-c_2)^2 k^2 - 3 \alpha \epsilon^2 (1-c_2) |k| - \alpha \epsilon^3 
\gtrsim |k|^3,
\end{align*}
for $\epsilon$ small. 
Then the supremum is finite for $s_1 - s < s +2 - 3b - \nu_{c_1}$ by the same reasoning as before.  

\noindent\textbf{Case 4.} $kn = 0$ 

The bound is immediate in this case. 

\subsection{Proof of Proposition \ref{R2}}

As in the previous proof, it suffices to show that the supremum of the following quantity is finite: 
\begin{align*} 
 \lb k \rb ^{2 + 2 s_1} \sum_{k_1, k_2}^* \frac{\lb k_1 \rb^{-2s} \lb k_2 \rb^{-2s} \lb k - k_1 - k_2 \rb ^{-2s} |k_1 + k_2 | ^2}{(\alpha k^3 - (k_1+k_2)^2 - \alpha(k-k_1 - k_2)^3)^2 \lb \alpha(k^3 - k_1^3 - k_2^3 - (k - k_1 - k_2)^3)\rb ^{2 -2b}}.
\end{align*}
We will work with the equivalent supremum
\begin{align*}
\; \sup_k \;  \lb k \rb ^{2 + 2 s_1} \sum_{\substack{k_1 \\ n \neq 0} }^* \frac{\lb k_1 \rb^{-2s} \lb n- k_1 \rb^{-2s} \lb n-k \rb ^{-2s}}{(n-d_1k)^2(n-d_2k)^2 \lb (k - k_1)(k + k_1 - n)n \rb  ^{2 -2b}},
\end{align*}
which results from changing variables $k_2 \mapsto n-k_1$ and canceling a factor of $n^2$ from the quotient. 

\noindent\textbf{Case 1.} $k_1 = k$ 

In this case, the supremum becomes 
\begin{align*}
\sup _k \lb k \rb ^{2+ 2s_1 - 2s} \sum_{n \neq 0}^* \frac{\lb k-n \rb ^{-4s}}{(n-d_1k)^2 (n-d_2k)^2}. 
\end{align*}
Repeat the arguments from Case 1 of the $R_1$ estimate to show that the supremum is finite if $s_1 -s \leq 1$ and $s-s_1 < 2s - \nu_d - 1$. 

\noindent\textbf{Case 2.} $n - k_1 = k$ 

The supremum becomes 
\begin{align*}
\; \sup_k \;  \lb k \rb ^{2 + 2 s_1 - 2s} \sum_{n \neq 0}^* \frac{\lb n-k \rb^{-4s}}{(n-d_1k)^2(n-d_2k)^2 },
\end{align*}
which is the same as that in Case 1. 

\noindent\textbf{Case 3.} $(k-k_1)(k+ k_1 - n)n \neq 0$ 

In this case, the supremum is bounded by 
\begin{align*}
\; \sup_k \;  \lb k \rb ^{2 + 2 s_1} \sum_{\substack{k_1 \\ n \neq 0} }^* \frac{\lb k_1 \rb^{-2s} \lb n- k_1 \rb^{-2s} \lb n-k \rb ^{-2s}}{(n-d_1k)^2(n-d_2k)^2 \lb k - k_1\rb ^{2 - 2b} \lb k + k_1 - n \rb^{2 -2b} \lb n \rb  ^{2 -2b}}.
\end{align*}

\noindent\textbf{Case 3.1.} $kn > 0$ 

Here $|n - d_2k| > k$, so the supremum is bounded by  
\begin{align*}
\; \sup_k \;  \lb k \rb ^{2 s_1} \sum_{\substack{k_1 \\ n \neq 0} }^* \frac{\lb k_1 \rb^{-2s} \lb n- k_1 \rb^{-2s} \lb n-k \rb ^{-2s}}{(n-d_1k)^2\lb k - k_1\rb ^{2 - 2b} \lb k + k_1 - n \rb^{2 -2b} \lb n \rb  ^{2 -2b}}.
\end{align*}

\noindent\textbf{Case 3.1a.} $|n-d_1k| \geq \epsilon |k|$ 

In this case we have the bound
\begin{align*}
&\; \sup_k \;  \lb k \rb ^{2 s_1 - 2} \sum_{\substack{k_1 \\ n \neq 0} } \lb k_1 \rb^{-2s} \lb n- k_1 \rb^{-2s} \lb n-k \rb ^{-2s} 
\lesssim \; \sup_k \;  \lb k \rb ^{2s_1 -s - 2}.
\end{align*}
This is finite if $s_1 -s \leq 1$. 

\noindent\textbf{Case 3.1b.} $\epsilon \leq |n - d_1 k | < \epsilon |k|$ 

Note that $|n| \in [(d_1 - \epsilon) |k|, (d_1 + \epsilon)|k|]$. Thus for $\epsilon$ small, $|n - k| \gtrsim |k|$. The supremum is finite when $s_1 -s \leq 3 -3b$: 
\begin{align*}
\; \sup_k \;  \lb k \rb ^{2 s_1 - 2s - 2 + 2b } \sum_{\substack{k_1 \\ n \neq 0} } \frac{\lb k_1 \rb^{-2s} \lb n- k_1 \rb^{-2s}}{\lb k - k_1\rb ^{2 - 2b} \lb k + k_1 - n \rb^{2 -2b}} 
&\lesssim \; \sup_k \;  \lb k \rb ^{2 s_1 - 2s - 4 + 4b } \sum_{k_1  } \frac{\lb k_1 \rb^{-2s}}{\lb k - k_1\rb ^{2 - 2b}} \\
&\lesssim \; \sup_k \;  \lb k \rb ^{2 s_1 - 2s - 6 + 6b } 
<  \infty
\end{align*}

\noindent\textbf{Case 3.1c.} $|n - d_1k| < \epsilon$ 

The supremum can be estimated by 
\begin{align*}
&\; \sup_k \;  \lb k \rb ^{2 s_1 + 2 + 2\nu_{d_1} + 2 \epsilon_0  - 2s - (2 - 2b)} \sum_{\substack{k_1 \\ (n \simeq d_1k)} } \frac{\lb k_1 \rb^{-2s} \lb n- k_1 \rb^{-2s}}{\lb k - k_1\rb ^{2 - 2b} \lb k + k_1 - n \rb^{2 -2b}}.
\end{align*}
If all four factors in the summation are of order at least $|k|$, this is easy to estimate. Furthermore, if any one factor in the summation is of order $\ll |k|$, then the other three factors are all $\gtrsim |k|$. This implies that the sum over $k_1$ can always be controlled by a sum of the form 
\[ \lb k \rb ^{-2(2-2b)-2s}\sum_{m} \lb m \rb ^{-2s} \lesssim  \lb k \rb ^{-2(2-2b)-2s}, \]
which means that the supremum is finite whenever
\[  2 s_1 + 2 + 2\nu_{d_1} + 2 \epsilon_0  - 4s - 3(2 - 2b) < 0, \] 
which holds when $ s_1 - s < s + \frac12 - \nu_{d_1}$. 

\noindent\textbf{Case 3.2.} $kn < 0$ 

Cancel $\lb k \rb ^2$ with $(n - d_1k)^2$ and repeat the arguments from Case 3.1. 

\subsection{Proof of Proposition \ref{R3}}

Decompose $R_3$ into two sums based on whether or not $k_1 + k_2 $ is zero:
\begin{align*}
R_3(u,v)_k &= ik \sum_{\substack{k_1 + k_2 + k_3 = k \\ k_1 \neq 0}}^* \frac{u_{k_1}u_{k_2}v_{k_3}(k_2+k_3)}{\alpha k^3 - k_1^3 - \alpha(k_2+k_3)^3} \\
&= ik  \sum_{\substack{k_1 + k_2 + k_3 = k \\ k_1 + k_2 \neq 0 \\ k_1 \neq 0}}^* \frac{u_{k_1}u_{k_2}v_{k_3}(k_2+k_3)}{\alpha k^3 - k_1^3 - \alpha(k_2+k_3)^3}  +ik v_k\sum_{k_1 \neq 0 }^* \frac{u_{k_1}u_{-k_1}(k-k_1)}{\alpha k^3 - k_1^3 - \alpha(k-k_1)^3} \\
&= ik \sum_{\substack{k_1 + k_2 + k_3 = k \\ k_1 + k_2 \neq 0 \\ k_1 \neq 0}}^* \frac{u_{k_1}u_{k_2}v_{k_3}(k_2+k_3)}{\alpha k^3 - k_1^3 - \alpha(k_2+k_3)^3}  \\
&+ ikv_k \sum_{k_1 > 0 }^* |u_{k_1}|^2 \left[\frac{k-k_1}{\alpha k^3 - k_1^3 - \alpha(k-k_1)^3} + \frac{k+k_1}{\alpha k^3 + k_1^3 - \alpha(k+k_1)^3}\right]\\
& = \mathrm{I} + \mathrm{II}.
\end{align*}

To bound $\mathrm{II}$ in $X^{s_1, b-1}_\alpha$, note that the bracketed sum is equal to 
\begin{align*}
\frac{2(1-\alpha)k_1^4}{k_1^2(k_1 - d_1k)(k_1-d_2k)(k_1 + d_1k)(k_1 + d_2k)} 
\end{align*}
and by an application of Cauchy-Schwartz and Young's inequalities, it suffices to show that 
\begin{align*}
\; \sup_k \;  \lb k \rb ^{2 + 2s_1 - 2s} \sum_{k_1 > 0 } ^* \frac{\lb k_1 \rb ^{4-4s}}{(k_1 - d_1k)^2(k_1-d_2k)^2(k_1 + d_1k)^2(k_1 + d_2k)^2} < \infty.
\end{align*}

\noindent\textbf{Case 1.} $k > 0$ 

In this case, $|k_1 - d_2k|, |k_1 + d_1 k| > k_1, k$ and the supremum can be estimated by 
\[ \; \sup_k \;  \lb k \rb ^{2s_1 - 2s} \sum_{k_1 > 0 } ^* \frac{\lb k_1 \rb^{2-4s}}{(k_1 - d_1k)^2 (k_1 + d_2k)^2}. \]

\noindent\textbf{Case 1.1.} $|k_1 - d_1k|, |k_1 + d_2k| \geq \epsilon k$ 

Here we have the estimates $|k_1| \leq \Big[(|d_1| + \epsilon)/\epsilon \Big] |k_1 - d_1k|$ and $|k_1| \leq \Big[(|d_2|+ \epsilon)/\epsilon \Big]|k_1 + d_2k|$. The supremum can thus be bounded by 
\[ \; \sup_k \;  \lb k \rb ^{2s_1 - 2s - 2} \sum_{k_1 > 0 } \lb k_1 \rb^{-4s} < \infty\]
for $s_1 -s \leq 1$. 

\noindent\textbf{Case 1.2.} $|k_1 - d_1 k| \geq \epsilon k$, $\epsilon \leq |k_1 + d_2k| < \epsilon k$ (or vice versa)

In this case, note that $k_1 \geq (d_1-\epsilon)k$ and bound the supremum by 
\begin{align*} 
&\; \sup_k \;  \lb k \rb ^{2s_1 - 2s - 1} \sum_{k_1 \gtrsim k } \lb k_1 \rb^{-4s +1}
\lesssim \; \sup_k \;  \lb k \rb ^{2s_1 - 2s - 1 - 4s + 2} 
< \infty
\end{align*}
when $s_1 -s \leq 2s - \frac12$. 

\noindent\textbf{Case 1.3.} $|k_1 - d_1k| \geq \epsilon k$, $|k_1 + d_2k| < \epsilon$ (or vice versa)

There is only one term in the sum in this case since $k_1 \simeq -d_2k$. The supremum can be bounded by 
\[ \; \sup_k \;  \lb k \rb ^{2s_1 -6s + 2 + 2 \nu_{d_2} + 2 \epsilon_0}, \]
which is finite when $s_1 -s < 2s - 1 - \nu_{d_2}$. 

This exhausts the cases with $k > 0$ since by choosing $\epsilon$ small, we may ensure that $|k_1 -d_1k| \leq \epsilon k $ and $|k_1 + d_2k| \leq \epsilon k$ cannot occur simultaneously.

\noindent\textbf{Case 2.} $k < 0$ 

Note that $|k_1 -d_1 k | , |k_1 + d_2k| > k_1, |k|$ and proceed as in Case 1. This completes the proof of the estimate for $\mathrm{II}$. \\  

To complete the proof, we must bound $\mathrm{I}$ in $X^{s_1, b-1}_\alpha$. As before, it suffices to show that the following supremum is finite:
\[ \; \sup_k \;  \lb k \rb ^{2 + 2s_1} \sum_{\substack{k_1 \neq 0, k_2 \\ k_1 + k_2 \neq 0}} \frac{ \lb  k_1 \rb ^{-2s} \lb k_2 \rb ^{-2s} \lb k-k_1-k_2 \rb^{-2s} |k-k_1|^2}{(\alpha k^3 - k_1^3 - \alpha (k-k_1)^3)^2 \lb \alpha k^3 - k_1^3 - k_2^3 - \alpha(k -k_1- k_2)^3 \rb ^{2 -2b}}, \]
or equivalently
\[ \; \sup_k \;  \lb k \rb ^{2 + 2s_1} \sum_{\substack{k_1 \neq 0 \\n \neq 0}} \frac{ \lb  k_1 \rb ^{-2s - 2} \lb n-k_1 \rb ^{-2s} \lb n-k \rb^{-2s} |k-k_1|^2}{(k_1 - d_1k)^2(k_1 - d_2k)^2 \lb \alpha k^3 - k_1^3 - (n-k_1)^3 - \alpha(k-n)^3 \rb ^{2 -2b}}. \]

\noindent\textbf{Case 1.} $|k_1 - d_1k| < \epsilon$ or $|k_1 - d_2k| < \epsilon $

Assume that $|k_1 - d_1k| < \epsilon$. The other case is parallel. Note that  
\[ |k_1 -d_2k| \geq (d_1 - d_2)|k| - |k_1 - d_1k| \]
so that $|k_1 - d_2k| \gtrsim |k|$. 
Also we have $|k_1 - k| \leq |k_1 -d_1k| + |d_1 - 1||k| \lesssim |k|$. 
The supremum may thus be bounded by 
\begin{align*}
\; \sup_k \;  \lb k \rb ^{2s_1 + 2 + 2 \nu_{d_1} + 2 \epsilon_0 - 2s} \sum_{\substack{n \neq 0 \\ (k_1 \simeq d_1 k)}} \frac{\lb n-k_1\rb^{-2s} \lb n- k \rb ^{-2s}}{\lb \alpha k^3 - k_1^3 - (n-k_1)^3 - \alpha (k-n)^3 \rb ^{2-2b}}.
\end{align*}

\noindent\textbf{Case 1.1.} $|n-k_1| \geq \epsilon |k|$, $|n-k| \geq \epsilon |k|$ 

In this case, the supremum is bounded by 
\begin{align*}
\; \sup_k \;  \lb k \rb ^{2s_1 + 2 + 2 \nu_{d_1} + 2 \epsilon_0 - 6s} \sum_{\substack{n \neq 0 \\ (k_1 \simeq d_1 k)}} \lb \alpha k^3 - k_1^3 - (n-k_1)^3 - \alpha (k-n)^3 \rb ^{-(2-2b)}.
\end{align*}
This is finite when $2s_1 +2 + 2 \nu_{d_1} - 6s < 0$, i.e. when $s_1 -s < 2s - 1 - \nu_{d_1}$. 

\noindent\textbf{Case 1.2.} $|n - k| < \epsilon |k|$ 

Note that in this region $|n- k_1| \gtrsim |k|$ since
\[ |n-k_1| \geq (1-d_1)|k| - |k-n| - |k_1 - d_1k| > (1-d_1 - \epsilon)|k| - \epsilon.\]

Write 
\[ n - k = \mu k \text{  for some  } |\mu| < \epsilon, \qquad k_1 = d_1k + \delta \text{  for some  } |\delta | < \epsilon. \]
Then 
\begin{align*}
| \alpha k^3 - k_1^3 - (n-k_1)^3 - \alpha (k-n)^3 | = \left| \alpha k^3 - (d_1k + \delta)^3 - [(1-d_1 + \mu)k - \delta] ^3 + \alpha \mu^3 k^3 \right| \\
=\left|  k^3[\alpha  - 1  + 3d_1 - 3d_1^2 + \mathcal{O}(\epsilon)] + \mathcal{O}(\epsilon)k^2 +\mathcal{O}(\epsilon^2)k + \mathcal{O}(\epsilon^3) \right|  \gtrsim |k^3|.
\end{align*}
The supremum is bounded by
\begin{align*}
&\; \sup_k \;  \lb k \rb ^{2s_1 + 2 + 2 \nu_{d_1} + 2 \epsilon_0 - 4s} \sum_{\substack{n \neq 0 \\ (k_1 \simeq d_1 k)}} \frac{ \lb n- k \rb ^{-2s}}{\lb \alpha k^3 - k_1^3 - (n-k_1)^3 - \alpha (k-n)^3 \rb ^{2-2b}} \\
\lesssim & \; \sup_k \;  \lb k \rb ^{2s_1 + 2 + 2 \nu_{d_1} + 2 \epsilon_0 - 4s} \sum_{n \neq 0} \frac{ \lb n- k \rb ^{-2s}}{\lb k^3 \rb^{2-2b}} 
\lesssim  \; \sup_k \;  \lb k \rb ^{2s_1 + 2 + 2 \nu_{d_1} + 2 \epsilon_0 - 4s - 6 + 6b } <  \infty, 
\end{align*}
for $s_1 - s < s + 2 - 3b  - \nu_{d_1}$. 

\noindent\textbf{Case 1.3.} $|n - k_1| < \epsilon|k|$ 

In this case, $|n-k| \gtrsim |k|$. Write 
\begin{align*}
n-k = \mu k \text{  for some  } |\mu| \in [1/|k|, \epsilon] \qquad k_1 - d_1k = \delta \text{  for some  } |\delta| < \epsilon.
\end{align*}
Then 
\begin{align*}
|\alpha k^3 - k_1^3 -(n-k_1)^3 - \alpha(k-n)^3 |= |\alpha k^3 - (d_1k + \delta)^3 -[(1 - d_1 + \mu)k - \delta]^3 + \alpha  \mu^3 k^3 | \\
= \left|[ \alpha - 1 + 3d_1 - 3d_1^2 + \mathcal{O}(\epsilon)]k^3 + \mathcal{O}(\epsilon)k^2 + \mathcal{O}(\epsilon^2)k + \mathcal{O}(\epsilon^3) \right| \gtrsim |k^3|.
\end{align*}
Thus the supremum is finite when $s_1 -s < s + 2 - 3b - \nu_{d_1}$:
\begin{align*}
&\; \sup_k \;  \lb k \rb ^{2s_1 + 2 + 2 \nu_{d_1} + 2 \epsilon_0 - 4s} \sum_{\substack{n \neq 0 \\ (k_1 \simeq d_1 k)}} \frac{\lb n-k_1\rb^{-2s}}{\lb k^3 \rb ^{2-2b}} 
\lesssim &\; \sup_k \;  \lb k \rb ^{2s_1 + 2 + 2 \nu_{d_1} + 2 \epsilon_0 - 4s - 6 + 6b} 
< & \infty.
\end{align*}

\noindent\textbf{Case 2.} $\epsilon \leq |k_1 - d_1k| < \epsilon|k|$ or $\epsilon \leq |k_1 - d_2k| <\epsilon |k|$ 

Assume that $\epsilon \leq |k_1 - d_1k| < \epsilon|k|$; the other case is similar. 
Note that we have $|k_1 - k| \lesssim |k|$ and $|k_1 - d_2k| \gtrsim |k|$ so that the supremum is bounded by    
\begin{align*}
\; \sup_k \;  \lb k \rb ^{2s_1} \sum_{\substack{|k_1| \gtrsim |k| \\n \neq 0}} \frac{ \lb  k_1 \rb ^{-2s} \lb n-k_1 \rb ^{-2s} \lb n-k \rb^{-2s} }{ \lb \alpha k^3 - k_1^3 - (n-k_1)^3 - \alpha(k-n)^3 \rb ^{2 -2b}}.
\end{align*}

\noindent\textbf{Case 2.1.} $|n - k_1| \geq \epsilon|k|$, $|n-k| \geq \epsilon|k|$ 

Here the supremum is bounded for $s_1 - s \leq 2s - \frac12$: 
\begin{align*}
&\; \sup_k \;  \lb k \rb ^{2s_1 - 4s} \sum_{\substack{|k_1| \gtrsim |k| \\n \neq 0}} \frac{ \lb  k_1 \rb ^{-2s} }{ \lb \alpha k^3 - k_1^3 - (n-k_1)^3 - \alpha(k-n)^3 \rb ^{2 -2b}} \\
&\lesssim \; \sup_k \;  \lb k \rb ^{2s_1 - 4s} \sum_{|k_1| \gtrsim |k|}  \lb  k_1 \rb ^{-2s}  
\lesssim \; \sup_k \;  \lb k \rb^{2s_1 - 4s - 2s +1 } 
< \infty.
\end{align*}

\noindent\textbf{Case 2.2.} $|n-k| < \epsilon |k|$ 

In this case, notice that $|n-k_1| \geq |k_1 - k| - |n-k| \geq (1-d_1 - 2\epsilon )|k| \gtrsim |k|$ and write
\[ n - k = \mu k \text{  for some  } |\mu| < \epsilon \qquad  k_1 - d_1 k = \mu' k \text{  for some  } |\mu'| < \epsilon. \]

Then 
\begin{align*}
|\alpha k^3 - k_1^3 - (n-k_1)^3 - \alpha(k-n)^3| &= |k^3[\alpha - (d_1 + \mu')^3  - (1 - d_1 + \mu - \mu')^3  + \alpha \mu^3]  |\\
&= |k^3[(1-\alpha) + 3(d_1^2 - d_1) + \mathcal{O}(\epsilon)] | \gtrsim |k|^3.
\end{align*}

Thus the supremum is bounded by 
\begin{align*}
&\; \sup_k \;  \lb k \rb ^{2s_1 - 2s } \sum_{\substack{|k_1| \gtrsim |k| \\|n| \lesssim |k| }} \frac{ \lb  k_1 \rb ^{-2s} }{ \lb k^3 \rb ^{2 -2b}} 
\lesssim  \; \sup_k \;  \lb k \rb ^{2s_1 -2s + 1 - 3(2-2b) - 2s + 1} 
< \infty 
\end{align*}
if $s_1 -s \leq s + 2 - 3b$. \\

\noindent\textbf{Case 2.3.} $|n - k_1| < \epsilon |k|$ 

Here $|n - k| \gtrsim |k|$, so the supremum can be estimated as follows for $s_1 - s \leq s - \frac12$:
\begin{align*}
&\; \sup_k \;  \lb k \rb ^{2s_1 - 2s} \sum_{\substack{|k_1| \gtrsim |k| \\n \neq 0}} \frac{ \lb  k_1 \rb ^{-2s}}{ \lb \alpha k^3 - k_1^3 - (n-k_1)^3 - \alpha(k-n)^3 \rb ^{2 -2b}} 
\lesssim  \; \sup_k \;  \lb k \rb ^{2s_1 - 2s - 2s + 1} 
<  \infty.
\end{align*}

\noindent\textbf{Case 3.} $|k_1 -d_1k| \geq \epsilon|k|$, $|k_1 - d_2k | \geq \epsilon |k|$ 

In this case, we need to bound 
\begin{align*}
\; \sup_k \;  \lb k \rb ^{2s_1-2} \sum_{\substack{k_1 \neq 0 \\n \neq 0}} \frac{ \lb  k_1 \rb ^{-2s - 2} \lb n-k_1 \rb ^{-2s} \lb n-k \rb^{-2s} |k-k_1|^2}{\lb \alpha k^3 - k_1^3 - (n-k_1)^3 - \alpha(k-n)^3 \rb ^{2 -2b}}.
\end{align*}

\noindent\textbf{Case 3.1.} $|k_1| \geq \epsilon |k|$

Here we have $|k-k_1| \lesssim |k_1|$ so that for $s_1 - s \leq 1$, the supremum can be estimated by 
\begin{align*}
&\; \sup_k \;  \lb k \rb ^{2s_1-2} \sum_{\substack{k_1 \neq 0 \\n \neq 0}}\lb  k_1 \rb ^{-2s} \lb n-k_1 \rb ^{-2s} \lb n-k \rb^{-2s} 
\lesssim  \; \sup_k \;  \lb k \rb ^{2s_1 -2s - 2} 
<  \infty.
\end{align*}

\noindent\textbf{Case 3.2.} $|k_1| \leq \epsilon |k|$ 

Here the supremum may be bounded by 
\begin{align*}
\; \sup_k \;  \lb k \rb ^{2s_1} \sum_{\substack{k_1 \neq 0\\n \neq 0}} \frac{ \lb  k_1 \rb ^{-2s - 2} \lb n-k_1 \rb ^{-2s} \lb n-k \rb^{-2s}}{\lb \alpha k^3 - k_1^3 - (n-k_1)^3 - \alpha(k-n)^3 \rb ^{2 -2b}}.
\end{align*}

\noindent\textbf{Case 3.2a.} $|n-k_1| \leq \epsilon |k|$ 

Note that we have $|n - k| \geq |k_1 - k| - |k_1 - n| \geq (1-2\epsilon)|k|$ 
and write 
\begin{align*}
k_1 &= \mu_1 k \text{  for some  } |\mu_1| \in [1/|k|, \epsilon ] \\
n - k_1 &= \mu_2 k \text{  for some  } |\mu_2| \in [1/|k|, \epsilon] \\
\mu &= \mu_1 + \mu_2 \text{  where }  |\mu| \in [1/|k|, 2 \epsilon] . 
\end{align*}
The lower bounds on $\mu_1$ and $\mu_2$ are positive because of the mean zero assumption on $u$. The lower bound on $\mu$  comes from the fact that $n \neq 0$. With this notation,
\begin{align*}
\left| \alpha k^3 - k_1^3 - (n-k_1)^3 - \alpha (k- n )^3 \right| 
&= |\mu k^3| \left| (\alpha - 1) \mu^2 + 3 \alpha( 1 - \mu)  + 3 \mu_1 \mu_2 \right|  \gtrsim |k^2|. 
\end{align*}
Then the supremum is bounded by 
\begin{align*}
 \; \sup_k \;  \lb k \rb ^{2s_1 - 2s} \sum_{\substack{k_1 \neq 0\\|n| \lesssim |k|}} \frac{ \lb  k_1 \rb ^{-2s - 2}}{\lb k^2 \rb ^{2 -2b}} 
\lesssim  \; \sup_k \;  \lb k \rb ^{2s_1 -2s - 4 + 4b} \sum_{\substack{k_1 \neq 0\\|n| \lesssim |k|}} \lb k_1 \rb ^{-2s-2} \\
\lesssim  \; \sup_k \;  \lb k \rb ^{2s_1 -2s - 4 + 4b + 1} \sum_{k_1 \neq 0} \lb k_1 \rb ^{-2s-2} 
\lesssim  \; \sup_k \;  \lb k \rb ^{2s_1 -2s - 4 + 4b + 1},
\end{align*}
which is finite for $s_1 -s \leq \frac32 - 2b$. 

\noindent\textbf{Case 3.2b.} $|n - k| \leq \epsilon |k|$ 

In this case, note that $|n-k_1| \geq |k-k_1| - |k-n| \geq (1-2\epsilon)|k| \gtrsim |k|$ 
and write
\begin{align*}
k_1 = \mu_1 k \text{  for some  } |\mu_1| \in [1/|k|, \epsilon ], \qquad
n - k = \mu_2 k \text{  for some  } |\mu_2| \in [0, \epsilon] . 
\end{align*}
Then
\begin{align*}
\alpha \left|k^3 - k_1^3 - (n-k_1)^3 - \alpha(k-n)^3\right| &=| k^3|\left|\alpha  - \mu_1^3 - (1 - \mu_1 + \mu_2)^3 +\alpha \mu_2^3 \right| \\
&= |k^3| \left| 1- \alpha + \mathcal{O}(\epsilon)\right| \gtrsim |k^3|.
\end{align*}
Thus for $s_1 -s \leq \frac52 - 3b$, the supremum is bounded:
\begin{align*}
\; \sup_k \;  \lb k \rb ^{2s_1 -2s} \sum_{\substack{k_1 \neq 0 \\ |n| \lesssim |k|}} \frac{\lb k \rb ^{-2s - 2}}{\lb k^3 \rb ^{2 - 2b}} 
\lesssim \; \sup_k \;  \lb k \rb ^{2s_1 -2s - 6 + 6b + 1 } \sum_{k_1 \neq 0 } \lb k_1 \rb ^{-2s - 2} \\
\lesssim \; \sup_k \;  \lb k \rb ^{2s_1 -2s - 6 + 6b + 1 } < \infty.
\end{align*} 


\section*{Acknowledgments}

The author would like to thank Nikolaos Tzirakis for many helpful discussions and comments.



\begin{thebibliography}{10}

\bibitem{AK}
A.~R.~Adem, C.~M.~Khalique, \emph{Symmetry reductions, exact solutions and conservation laws of a new coupled KdV system}, Commun. Nonlinear Sci. Numer. Simul. \textbf{17} (2012), no. 9, p. 3465-3475. 

\bibitem{ACW}
J.~M.~Ash, J.~Cohen, G.~Wang, \emph{On strongly interacting internal solitary waves}, J. Fourier Anal. Appl. \textbf{2} (1996), no. 5, p. 507-517. 

\bibitem{Arn}
V. ~Arnold, \emph{Geometrical methods in the theory of ordinary differential equations}. 2nd ed. Fundamental Principles of Mathematical Sciences 50. Springer-Verlag, New York, 1988.

\bibitem{BIT}
A.~Babin, A.~Ilyin, E.~Titi, \emph{On the regularization mechanism for the periodic Korteweg-de Vries equation}, Comm. Pure Appl. Math. \textbf{64} (2011), no. 5, p. 591-648. 

\bibitem{Biello} 
J.~Biello, \emph{Nonlinearly coupled KdV equations describing the interaction of equatorial and midlatitude Rossby waves}, Chin. Ann. Math. Ser. B \textbf{30} (2009), no. 5, p. 483-504. 

\bibitem{BPST}
J.~Bona, G.~Ponce, J.-C.~Saut, M.~Tom, \emph{A model system for strong interaction between internal solitary waves}, Comm. Math. Phys. \textbf{143} (1992), no. 2, p. 287-313. 

\bibitem{Bourg}
J.~Bourgain, \emph{Fourier transform restriction phenomena for certain lattice subsets and noninear evolution equations. Part II: The KdV equation}, GAFA \textbf{3} (1993), p. 209-262. 

\bibitem{Bourg2}
\rule[0.48ex]{3em}{0.14ex}\space, \emph{On the growth in time of higher Sobolev norms of smooth solutions of Hamiltonian PDE}, Int. Math. Res. Not. (1998), no. 5, p. 253-283. 

\bibitem{CR}
M.~Cabral, R.~Rosa, \emph{Chaos for a damped and forced KdV equation}, Phys. D \textbf{192} (2004), no. 304, p. 265-278. 



\bibitem{EMT} 
M.~B.~Erdo\u{g}an, J.~Marzuola, N.~Tzirakis, \emph{The structure of global attractors for dissipative Zakharov systems with forcing on the torus}, preprint 2013, 16 p. 

\bibitem{ET1}
M.~B.~Erdo\u{g}an, N.~Tzirakis, \emph{Global smoothing for the periodic KdV evolution}, Int. Math. Res. Not. (2013), no. 20, p. 4589-4614.

\bibitem{ET3} 
\rule[0.48ex]{3em}{0.14ex}\space, \emph{Long time dynamics for the forced and weakly damped KdV on the torus}, Commun. Pure Appl. Anal. \textbf{12} (2013), no 6, p. 2669-2684. 

\bibitem{ET2}
\rule[0.48ex]{3em}{0.14ex}\space, \emph{Smoothing and global attractors for the Zakharov system on the torus}, Anal. PDE \textbf{6} (2013), no. 3, p. 723-750.

\bibitem{Feng}
X.~Feng, \emph{Global well-posedness of the initial value problem for the Hirota-Satsuma system}, Manuscripta Math. \textbf{84} (1994), no. 3-4, p.361-378. 

\bibitem{GG} 
J.~A.~Gear, R.~Grimshaw, \emph{Weak and strong interactions between internal solitary waves}, Stud. Appl. Math. \textbf{70} (1984), no. 3, p. 235-258. 

\bibitem{Ghid}
J.~M.~Ghidaglia, \emph{Weakly damped forced Korteweg-de Vries equations behave as a finite dimensional dynamical system in the long time}, J. Diff. Eqs. \textbf{74} (1988), p. 369-390. 

\bibitem{GTV}
J.~Ginibre, Y.~Tsutsumi, G.~Velo, \emph{On the cauchy problem for the Zakharov system}, J. Funct. Anal \textbf{151} (1997), no. 2, p. 384-436. 

\bibitem{GST} 
Y.~Guo, K.~Simon, E.~Titi, \emph{On a nonlinear system of coupled KdV equations}, Commun. Math. Sci. (to appear), 32 p. 

\bibitem{HS}
J.~Satsuma, R.~Hirota, \emph{Soliton solutions of a coupled Kortewege-de Vries equation}, Phys. Lett. A \textbf{85} (1981), no. 8-9, p. 407-408. 

\bibitem{KST}
T.~Kappeler, B.~Schaad, P.~Topalov, \emph{Qualitative features of periodic solutions of KdV}, Commun. Part. Diff. Eqs. \textbf{38} (2013), no. 9, p. 1626-1673. 

\bibitem{KPV}
C.~Kenig, G.~Ponce, L.~Vega,  \emph{A bilinear estimate with applications to the KdV equation}, J. Amer. Math. Soc. \textbf{9} (1996), no. 2, p. 573-603. 

\bibitem{LP}
F.~Linares, M.~Panthee, \emph{On the Cauchy problem for a coupled system of KdV equations}, Commun. Pure Appl. Anal. \textbf{3} (2004), no. 3, p. 417-431.

\bibitem{MB} 
A.~Majda, J.~Biello, \emph{The nonlinear interaction of barotrophic and equatorial baroclinic Rossby waves}, J. Atmospheric Sci. \textbf{60} (2003), no. 15, p. 1809-1821. 

\bibitem{Oh2}
T.~Oh, \emph{Diophantine conditions in global well-posedness for coupled KdV-type system}, Electron. J. Differential Equations (2009), no. 52, 48 p. 

\bibitem{Oh}
\rule[0.48ex]{3em}{0.14ex}\space, \emph{Diophantine conditions in well-posedness theory of coupled KdV-type systems: Local theory}, Int. Math. Res. Not. (2009), no. 18, p. 3516-3556. 

\bibitem{Roth}
K.~Roth, \emph{Rational approximations to algebraic numbers}, Mathematika \textbf{2} (1955), p. 1-20. 

\bibitem{ST}
J.-C.~Saut, N.~Tzvetkov, \emph{On a model system for the oblique interaction of internal gravity waves}, Special issue for R. Temam's 60th birthday, M2AN Math. Model. Numer. Anal. \textbf{34} (2000), no. 2, p. 501-523. 

\bibitem{Staff}
G.~Staffilani, \emph{On the growth of high Sobolev norms of solutions for KdV and Schr\"{o}dinger equations}, Duke Math. J. \textbf{86} (1997), no. 1, p. 109.142. 

\bibitem{Tem}
R.~Temam, \emph{Infinite-dimensional dynamical systems in mechanics and physics}. 2nd ed. Applied Mathematical Sciences 68. Springer-Verlag, New York, 1997. 

\bibitem{Vod}
J.~Vodov\'{a}-Jahnov\'{a}, \emph{On symmetries and conservation laws of the Majda-Biello system}, Nonlinear Anal. Real World Appl. \textbf{22}  (2015), p. 148-154. 

\end{thebibliography}
\end{document}